\documentclass[11pt, reqno]{amsart}
\usepackage{amsfonts,amsmath,amsthm,amssymb,stmaryrd}
\usepackage{graphicx}\usepackage{bm}
\usepackage{esint}
\usepackage{hyperref}
\usepackage[RGB]{xcolor}
\usepackage{epstopdf}
\newtheorem{Lemma}{Lemma}[section]
\newtheorem{Theorem}{Theorem}[section]
\newtheorem{Definition}{Definition}[section]

\newtheorem{Remark}{Remark}[section]

\numberwithin{equation}{section} \allowdisplaybreaks

\usepackage[left=30 mm, right=30 mm,top=30 mm, bottom=30 mm]{geometry}

\def\G{{\mathcal G}}

\def\L{{\bf L}}

\def\D{{\mathcal D}}
\def\P{{\mathcal P}}

\def\dint{\int\!\!\int}
\def\ve{\varepsilon}

\def\vp{\varphi}

\def\Tilde{\widetilde}

\def\C{{\mathcal C}}

\def\N{{\mathcal N}}

\def\Q{{\mathcal Q}}

\def\ds{\displaystyle}
\def\sqr#1#2{\vbox{\hrule height .#2pt
\hbox{\vrule width .#2pt height #1pt \kern #1pt
\vrule width .#2pt}\hrule height .#2pt }}
\def\square{\sqr74}
\def\endproof{\hphantom{MM}\hfill\llap{$\square$}\goodbreak}

\def\bega{\begin{array}}
\def\enda{\end{array}}
\def\begi{\begin{itemize}}
\def\endi{\end{itemize}}

\def\K{{\mathcal K}}
\def\O{{\mathcal O}}

\def\R{I\!\!R}

\def\ov{\overline}
\def\Tilde{\widetilde}

\def\be{\begin{equation}}
\def\beq{\begin{equation}}
\def\bel{\begin{equation}\label}
\def\eeq{\end{equation}}

\def\R{\mathop{\mathbb R\kern 0pt}\nolimits}

\newcommand{\beno}{\begin{eqnarray*}}
\newcommand{\eeno}{\end{eqnarray*}}

\def\vp{\varphi}

\def\({\left(\begin{array}{cccccc}}
\def\){\end{array}\right)}

\def\({\left(\begin{array}{cccccc}}
\def\){\end{array}\right)}

\def\bes{\begin{eqnarray}}
\def\ees{\end{eqnarray}}

\newcommand{\bea}{\begin{eqnarray}}
\newcommand{\eea}{\end{eqnarray}}
\newcommand{\beann}{\begin{eqnarray*}}
\newcommand{\eeann}{\end{eqnarray*}}

\newcommand{\bp}{\begin{proof}}
\newcommand{\ep}{\end{proof}}

\begin{document}
\vskip 0.2cm

\title[Uniqueness of conservative solutions to  nonlinear variational wave equation]{\bf  Uniqueness of conservative solutions to a one-dimensional general quasilinear wave equation through variational principle}

\author[H. Cai]{Hong Cai}
\address{Hong Cai \newline
Department of Mathematics and Research Institute for Mathematics and Interdisciplinary Sciences, Qingdao University of Science and Technology, Qingdao, Shandong, P.R. China, 266061.}
\email{caihong19890418@163.com}

\author[G. Chen]{Geng Chen}
\address{Geng Chen \newline
Department of Mathematics, University of Kansas, Lawrence, KS 66045, USA.}
\email{gengchen@ku.edu}

\author[Y. Du]{Yi Du}
\address{Yi Du \newline
Department of Mathematics, JiNan University, Guangzhou, 510632, P. R. China.}
\email{duyidy@jnu.edu.cn}

\author[Y. Shen]{Yannan Shen}
\address{Yannan Shen \newline
Department of Mathematics, University of Kansas, Lawrence, KS 66045, USA.}
\email{yshen@ku.edu}


\maketitle

\begin{abstract}
{\small In this paper, we prove the uniqueness of energy conservative H\"older continuous weak solution to a general quasilinear wave equation by the analysis of characteristics. This result has no restriction on the size of solutions, i.e. it is a large data result.
}
 \bigbreak
\noindent


\end{abstract}

\section{Introduction}
\setcounter{equation}{0}

Consider a class of hyperbolic system of nonlinear wave equations that are derived from a variational principle whose action is a quadratic function of the derivatives of the field with coefficients depending on the field and the independent variables
\begin{equation}\label{field}
\delta\int A_{\mu\nu}^{ij}(\mathbf{x},u)\frac{\partial u^\mu}{\partial x_i}\frac{\partial u^\nu}{\partial x_j}d\mathbf{x}=0,
\end{equation}
where we use the summation convention. Here ${\bf x}\in \mathbb{R}^{d+1}$ are the space-time variables and $u:\mathbb{R}^{d+1}\to {\mathbb R}^n$ are the dependent variables. In this paper, we always assume the coefficients $A_{\mu\nu}^{ij}:\mathbb{R}^{d+1}\times \mathbb{R}^n\to \mathbb{R}$ are smooth and satisfy $A_{\mu\nu}^{ij}=A_{\nu\mu}^{ij}=A_{\mu\nu}^{ji}$.
The Euler-Lagrange equations associated with \eqref{field} are
\begin{equation}\label{EulerL}
\frac{\partial}{\partial x_i}\Big(A_{k\mu}^{ij}\frac{\partial u^\mu}{\partial x_j}\Big)=\frac{1}{2}\frac{\partial A_{\mu\nu}^{ij}}{\partial u^k}\frac{\partial u^\mu}{\partial x_i}\frac{\partial u^\nu}{\partial x_j}.
\end{equation}

System \eqref{field} has various physical backgrounds. See \cite{AH2007} for some backgrounds of this general system. In particular, system \eqref{field} has direct applications on nematic liquid crystals \cite{AH, BZ}, which will be introduced in the next part. System \eqref{field} is also realted to the $O(3)\ \sigma$-model, see the introduction of \cite{CCS}. Here the $O(3)\ \sigma$-model has applications on many physical areas, including the general relativity and Yang-Mills fields, \cite{RS}.

A particular physical example leading to \eqref{field} is the motion of a massive director field in a nematic liquid crystal. More precisely, a nematic crystal can be described, when we
ignore the motion of the fluid, by a director field of unit vectors $\mathbf{n}\in\mathbb{S}^2$ describing the orientation of rod-like molecules.
In the regime in which inertia effects
dominate viscosity, the propagation of the orientation waves in the director field is modeled by the least action principle (Saxton \cite{S})
\begin{equation}\label{ofdelta}
\delta\int\Big(\partial_t \mathbf{n}\cdot \partial_t \mathbf{n}-W(\mathbf{n},\nabla\mathbf{n}) \Big)\,d\mathbf{x}\,dt=0,\quad \mathbf{n}\cdot\mathbf{n}=1,
\end{equation}
where $W(\mathbf{n},\nabla\mathbf{n})$ is the well-known Oseen-Franck potential energy density,
\begin{equation*}
W(\mathbf{n},\nabla\mathbf{n})=K_1|\mathbf{n}\times(\nabla\times\mathbf{n})|
+K_2(\nabla \cdot \mathbf{n})^2+K_3(\mathbf{n}\cdot\nabla\times\mathbf{n})^2.
\end{equation*}
Here the positive constants $K_1, K_2$ and $K_3$ are elastic constants of the liquid crystal. Since $W(\mathbf{n},\nabla\mathbf{n})$ is a quadratic function of $\nabla\mathbf{n}$, with coefficients depending on $\mathbf{n}$, this variational principle is of the form \eqref{field}.

The simplest class of solutions for orientation waves in \eqref{ofdelta} of planar deformations depending on a single space variable $x$. The director field then has the special form
$
\mathbf{n}=\cos u(t,x)\mathbf{e}_x+\sin u(t,x)\mathbf{e}_y,
$
where the dependent variable $u \in \mathbb{R}$ measures the angle of director field to $x$-direction, and $\mathbf{e}_x$ and $\mathbf{e}_y$ are the coordinate vectors in $x$ and $y$ directions, respectively. In this case
the functional $W(\mathbf{n},\nabla\mathbf{n})$ vastly simplifies to $W(\mathbf{n},\nabla\mathbf{n})=(K_1\cos^2u+K_2 \sin^2u)u_x^2$
and $|\mathbf{n}_t|^2=u_t^2$.
Then the dynamics are described by the variational principle
\begin{equation*}
\delta \int[u_t^2-c^2(u)u_x^2]\,dx\,dt=0,
\end{equation*}
with the wave speed $c$ given by $c^2(u)=K_1\cos^2 u+K_2\sin^2 u$.
Thus, the Euler-Lagrange equation for this variational principle results one representative example of variational wave equation
\begin{equation}\label{vwe}
u_{tt}-c(u)[c(u)u_x]_x=0.
\end{equation}
Because of strong nonlinearity, the solution for the initial value problem of \eqref{vwe} generically forms finite time cusp singularity \cite{BC,BHY,GHZ}. Hence, in general, we have to consider weak solutions, such as the energy conservative ($H^1$) solution considered in this paper. The low regularity makes the study on the global well-posedness very difficult.

Currently, the global well-posedness of conservative solution for
\eqref{vwe} has been fairly well understood after a sequence of papers.
The global existence of H\"older continuous energy conservative solution was established by Bressan and Zheng in \cite{BZ}. Later this result was extended to more general initial data in \cite{HR}, the case with damping in \cite{CZ12} and the variational wave system \eqref{ofdelta} with ${\bf n}\in{\mathbb S}^2$ in  \cite{CZZ12,ZZ10,ZZ11}. Also see the existence of dissipative solution with monotonic wave speed $c(\cdot)$ in \cite{BH,ZZ03}.

To select a unique solution after singularity formation, one needs to add an additional admissible condition, such as the energy conservative condition. In \cite{BCZ}, uniqueness of energy conservative solution has been established by Bressan, Chen and Zhang. Furthermore, Bressan and Chen in \cite{BC} proved a generic regularity result, which serves as a key part in the study of Lipschitz continuous dependence later in \cite{BC2015} by Bressan and Chen, where the solution flow was proved to be Lipschitz continuous on a new Finsler type optimal transport metric. In fact, the solution flow fails to be Lipschitz continuous under existing metrics, such as Sobolev metrics or Wasserstein metrics.
Later, the uniqueness result has been extended to  systems \eqref{ofdelta} in \cite{CCD}. 

\vskip 0.3cm
In this paper, we are interested in a more general model of variational wave equation: \eqref{field} with  $n=1$ and $d=1$. Then the Euler-Lagrange equation \eqref{EulerL} reads that
\begin{equation}\label{EL}
(A^{11}u_t+A^{12}u_x)_t+(A^{12}u_t+A^{22}u_x)_x=\frac{1}{2}
\big(\frac{\partial A^{11}}{\partial u}u_t^2+2\frac{\partial A^{12}}{\partial u}u_t u_x+\frac{\partial A^{22}}{\partial u}u_x^2\big).
\end{equation}
Moreover, we assume the coefficients in \eqref{EL} satisfy
\begin{equation*}
(A^{ij})_{2\times 2}=\left(
  \begin{array}{cc}
   \alpha^2 & \beta \\
   \beta & -\gamma^2 \\
\end{array}
\right)(x,u),
\end{equation*}
then equation \eqref{EL} exactly gives the following nonlinear variational wave equation, on which we focus in this paper
\begin{equation}\label{vwl}
(\alpha^2 u_t+\beta u_x)_t+(\beta u_t-\gamma^2 u_x)_x=\alpha\alpha_u u_t^2+\beta_u u_tu_x-\gamma\gamma_u u_x^2.
\end{equation}
Here the variable $t\geq 0$ is time, and $x\in\mathbb{R}$ is the spatial coordinate.
We consider the initial data satisfying
\begin{equation}\label{ID}
u(0,x)=u_0(x)\in H^1,\quad
u_t(0,x)=u_1(x)\in L^2.
\end{equation}

Still due to the singularity formation, for the general case, one needs to consider weak solutions. The initial condition \eqref{ID} is corresponding to the finite initial energy case. Because the solution has finite speed of propagation,  \eqref{ID} is the most interesting general initial data.

The coefficients $\alpha=\alpha(x,u),\beta=\beta(x,u),\gamma=\gamma(x,u)$ are smooth functions on $x$ and $u$, satisfying that, for any $z=(x,u)\in\mathbb{R}^2$, there exist positive constants $\alpha_1,\alpha_2,\beta_2,\gamma_1$ and $\gamma_2$, such that
\begin{equation}\label{con}
\begin{cases}
0<\alpha_1\leq \alpha(z)\leq\alpha_2,~~|\beta(z)|\leq \beta_2,~~0<\gamma_1\leq\gamma(z)\leq\gamma_2,\\
\displaystyle\sup_z\{|\nabla\alpha(z)|,|\nabla\beta(z)|,|\nabla\gamma(z)|\}<\infty.
\end{cases}\end{equation}
In this paper, subscripts $t$, $x$ or $u$ represent partial derivatives with respect to $t$, $x$ or $u$.
Then system \eqref{vwl} is strictly hyperbolic with two eigenvalues
\begin{equation}\label{lambda}
\lambda_{-}:=\frac{\beta-\sqrt{\beta^2+\alpha^2\gamma^2} }{\alpha^2}<0,\qquad
\lambda_{+}:=\frac{\beta+\sqrt{\beta^2+\alpha^2\gamma^2} }{\alpha^2}>0.
\end{equation}
We will always call waves in the families of $\lambda_-$ and $\lambda_+$ as backward and forward waves, respectively. By \eqref{con}, $-\lambda_{-}$ and $\lambda_{+}$ are both smooth, bounded and uniformly positive.

It is easy to see that the equation \eqref{vwl} is a general quasilinear wave equation including \eqref{vwe} as an example, when $\alpha=1, \beta=0$ and $\gamma=c(u)$.

Solutions of \eqref{vwl}--\eqref{ID} may form finite time cusp singularity, see examples in \cite{BC,BHY,GHZ}. The existence of global-in-time energy conservative H\"older continuous (weak) solution has been established by Hu in \cite{H}, applying the transformation of coordinates method first used in \cite{BZ}.
In this paper, we would like to address the issue of uniqueness for conservative solution to the nonlinear variational wave equation \eqref{vwl}--\eqref{ID}. Our proof is based on the framework established by Bressan, Chen and Zhang in \cite{BCZ}. For \eqref{vwl}, we face a much more involved case than the variational wave equation \eqref{vwe}. As a result, many estimates and constructions in this paper are considerably complicated.

Some estimates in this paper will also serve as crucial preparations in another paper \cite{CCS} addressing the Lipschitz continuous dependence of solution. In fact, in \cite{CCS}, we will construct a new distance  which renders Lipschitz continuous the solution flow of \eqref{vwl} constructued in \cite{H}, although the solution does not depend continuously on the initial data with respect to the natural Sobolev space corresponding to energy. Because the uniqueness result in the current paper rules out the possibility of constructing a different solution through any method other than the one used in \cite{H}, one can fairly well concludes the global well-posedness of conservative solution of  \eqref{vwl}--\eqref{ID} by the current paper and \cite{CCS,H}.

The outline of the paper is as follows. In Section 2, after reviewing the existence result of a conservative solution to \eqref{vwl}--\eqref{ID}, the main uniqueness result in this paper will be introduced. In Section 3,
we study the existence and uniqueness of characteristic in each direction. In Section 4, with some other auxiliary
variables introduced associated to a given conservative solution, we prove that these variables satisfy a particular semi-linear system. Then the uniqueness of conservative solution in the original variable $u$ will be concluded.

\section{Main result}\label{sec_mainresult}
In this section, we review the global existence of an energy-conservative weak solution to the Cauchy problem \eqref{vwl}--\eqref{ID}, c.f. \cite{H}. Our main uniqueness
result will be stated at the end of this section.

\subsection{Existing existence result}\label{sec_rew}
\setcounter{equation}{0}
In this part, we recall that the problem \eqref{vwl}--\eqref{ID} has a weak solution which conserves the total energy. For more details, the readers can refer to \cite{H}.

We first review the global existence theorem in \cite{H}. One can also find the definition of weak solution inside this theorem.
\begin{Theorem} [Global existence \cite{H}] \label{thm_ex}
Let the condition \eqref{con} be satisfied, then the Cauchy problem \eqref{vwl}--\eqref{ID} admits a global {\bf weak solution} $u=u(t,x)$ defined for all $(t,x)\in\mathbb{R}^+\times\mathbb{R}$, as follows:
\begin{itemize}
\item[(i)] In the $t$-$x$ plane, the function $u(t,x)$ is locally H\"older continuous
with exponent $1/2$.  The function $t\mapsto u(t,\cdot)$ is
continuously differentiable as a map with values in $L^p_{\rm
loc}$, for all $1\leq p<2$. Moreover, it is Lipschitz continuous with respect to
(w.r.t.)~the $L^2$ distance, that is, there exists a constant $L$ such that
\begin{equation*}
 \big\|u(t,\cdot)-u(s,\cdot)\big\|_{L^2}
 \leq L\,|t-s|,
\end{equation*}
for all $t,s\in\mathbb R^+$.
\item[(ii)] The function $u(t,x)$ takes on the initial conditions in (\ref{ID})
pointwise, while their temporal derivatives hold in  $L^p_{\rm
loc}\,$ for $p\in [1,2)\,$.
\item[(iii)] The equations (\ref{vwl}) hold in distributional sense, that is
\begin{equation*}\label{wevwl}
\int\int\big[\varphi_t(\alpha^2 u_t+\beta u_x)+\varphi_x(\beta u_t-\gamma^2 u_x)+\varphi(\alpha \alpha_u u_t^2+\beta_u u_t u_x-\gamma\gamma_u u_x^2)\big]\,dx\,dt=0
\end{equation*}
for all test function $\varphi\in
C^1_c(\mathbb R^+\times \mathbb R)$.
\end{itemize}
\end{Theorem}

Denote wave speeds as
$$c_1:=\alpha\lambda_-<0,\quad {\rm and}\quad c_2:=\alpha\lambda_+>0,$$
 and the Riemann variables as
\begin{equation*}\label{R-S}
 R:=\alpha u_t+c_2 u_x,\quad S:=\alpha u_t+c_1 u_x.\end{equation*}
By \eqref{lambda}, the wave speeds satisfy that $-c_1(x,u)$ and  $c_2(x,u)$ are both smooth on $x$ and $u$, bounded and uniformly positive.
For a smooth solution of  (\ref{vwl}), the variables $R$ and $S$ satisfy
\begin{equation} \label{R-S-eqn}
\begin{cases}
\alpha(x,u) R_t+c_1(x,u)R_x=a_1 R^2-(a_1+a_2)RS+a_2 S^2+c_2bS-d_1 R,\\
\alpha(x,u)  S_t+c_2(x,u)S_x=-a_1 R^2+(a_1+a_2)RS-a_2 S^2+c_1bR-d_2 S,\\
\displaystyle u_t=\frac{c_2S-c_1R}{\alpha(c_2-c_1)} \quad{\text or} \quad u_x=\frac{R-S}{c_2-c_1}.
 \end{cases}
 \end{equation}
 where
\begin{equation*}
\begin{split}
&a_i=\frac{c_i\partial_u\alpha-\alpha\partial_u c_i}{2\alpha(c_2-c_1)},\quad b=\frac{\alpha\partial_x(c_1-c_2)+(c_1-c_2)\partial_x\alpha}{2\alpha(c_2-c_1)},\\
&d_i=\frac{c_2\partial_x c_1-c_1\partial_x c_2}{2(c_2-c_1)}+\frac{\alpha\partial_x c_i-c_i\partial_x\alpha}{2\alpha},\quad(i=1,2),
\end{split}
\end{equation*}
and $\partial_x$ and $\partial_u$ denote partial derivatives with respect to $x$ and $u$, respectively.

Multiplying the first equation in \eqref{R-S-eqn} by $R$ and the second one by $S$, one has the balance laws for energy densities in two directions, namely
\begin{equation}\label{balance}
\begin{cases}
\displaystyle(\Tilde{R}^2)_t+(\frac{c_1}{\alpha}\Tilde{R}^2)_x=G,\\
\displaystyle(\Tilde{S}^2)_t+(\frac{c_2}{\alpha}\Tilde{S}^2)_x=-G,
\end{cases}
\end{equation}
where \begin{equation*}
\begin{split}
&\Tilde{R}^2=\frac{-c_1}{c_2-c_1}R^2,\quad \Tilde{S}^2=\frac{c_2}{c_2-c_1}S^2,\quad{\rm and}\\
&G=\frac{2c_2a_1}{\alpha(c_2-c_1)}R^2 S-\frac{2c_1a_2}{\alpha(c_2-c_1)}RS^2-\frac{2c_1c_2b}{\alpha(c_2-c_1)}RS,
\end{split}\end{equation*}
which indicates the following conserved quantities
\begin{equation*}
\alpha^2 u_t^2+\gamma^2u_x^2=\Tilde{R}^2+\Tilde{S}^2.\end{equation*}

Now we define the energy conservation as following.

\begin{Definition}[Energy conservation \cite{H}]\label{def_ec}
The weak solution defined in Theorem \ref{thm_ec} is {\bf energy conserved} (or {\bf conservative}), if
there exist two  families of positive Radon measures
on the real line: $\{\mu_-^t\}$ and $\{\mu_+^t\}$, depending continuously
on $t$ in the weak topology of measures, with the following properties.
\begi
\item[(i)] At every time $t$ one has
\begin{equation*}
\mu_-^t(\mathbb {R})+\mu_+^t(\mathbb {R})~=~E_0:=~
\int_{-\infty}^\infty \Big[\alpha^2\big(x,u_0(x)\big) u_1^2(x)+ \gamma^2\big(x,u_0(x)\big) u_{0,x}^2(x) \Big]\, dx \,,\end{equation*}

\item[(ii)] For each $t$, the absolutely continuous parts of $\mu_-^t$ and
$\mu_+^t$ with respect to the Lebesgue measure
have densities  respectively given  by
\begin{equation*}
\Tilde{R}^2=\frac{-c_1}{c_2-c_1}(\alpha u_t+c_2 u_x)^2,\qquad
\Tilde{S}^2=\frac{c_2}{c_2-c_1}(\alpha u_t+c_1 u_x)^2.
\end{equation*}
\item[(iii)]  For almost every $t\in\mathbb {R}^+$, the singular parts of $\mu^t_-$ and $\mu^t_+$
are concentrated on the set where $\partial_u \lambda_-=0$ or $\partial_u \lambda_+=0$.

\item[(iv)] The measures $\mu_-^t$ and $\mu_+^t$ provide measure-valued solutions
respectively to the balance laws
\bel{mbl}
\begin{cases}
w_t + (\frac{c_1}{\alpha}w)_x  = \frac{2c_2a_1}{\alpha(c_2-c_1)}R^2 S-\frac{2c_1a_2}{\alpha(c_2-c_1)}RS^2-\frac{2c_1c_2b}{\alpha(c_2-c_1)}RS, \\
z_t + (\frac{c_2}{\alpha}z)_x  =  - \frac{2c_2a_1}{\alpha(c_2-c_1)}R^2 S+\frac{2c_1a_2}{\alpha(c_2-c_1)}RS^2+\frac{2c_1c_2b}{\alpha(c_2-c_1)}RS.
\end{cases}
\eeq
\endi
\end{Definition}

Then another main result proved in \cite{H}, on energy conservation can be stated as follows.

\begin{Theorem}[Energy conservation \cite{H}]\label{thm_ec}
 Let the condition \eqref{con} be satisfied, then the solutions $u(t,x)$
constructed in Theorem \ref{thm_ex} are conservative in the sense of Definition \ref{def_ec}.
\end{Theorem}

Theorem \ref{thm_ec} implies that for  the above conservative weak solutions, the total energy represented by the sum $\mu_-+\mu_+$ is
conserved in time. This energy may only be concentrated
on a set of zero measure or at points where $\partial_u\lambda_-$ or $\partial_u\lambda+$ vanishes. In particular, if $\partial_u \lambda_\pm\not= 0$
for any $(x,u)$, then the set
$$\Big\{\tau;~E(\tau):=~\int_{-\infty}^\infty\Big[|
\alpha^2\big(x,u(\tau,x)\big) u_t^2(\tau,x)+
\gamma^2\big(x,u(\tau,x)\big) {u}_x^2(\tau,x)
\Big]\, dx  ~<~E_0\Big\}$$
has measure zero.

\subsection{Main result of this paper}
The goal of present paper is to understand whether the conservative solution to \eqref{vwl}--\eqref{ID} is unique. The result is stated below.
\begin{Theorem}[Uniqueness]\label{thm_un}
Let the condition \eqref{con} be satisfied,
the conservative weak solution, defined in Theorem \ref{thm_ex} and Definition \ref{def_ec}, to the Cauchy problem (\ref{vwl})--(\ref{ID}) is unique.
\end{Theorem}

We prove Theorem \ref{thm_un}, the uniqueness of conservative solutions $u=u(t,x)$, relying on the analysis of characteristics. This framework was first established in \cite{BCZ}.  Here the main difficulty in this uniqueness result arises from the low regularity of solution due to possible concentration of energy. In fact, when singularity forms as energy concentrates, solution might be only H\"older continuous. So the characteristic equations
\[
\frac{dx(t)}{dt}=\lambda_{\pm}\big(x(t),u(t,x(t))\big)
\]
whose right hand side are only  H\"older continuous on $t$,
might exist more than one solutions, after energy concentration. One needs to find a way to use the energy conservation law in its weak sense, to select a unique characteristic after energy concentration (step 1), then a unique solution of \eqref{vwl}--\eqref{ID} (step 2).

The underlying idea is to first introduce a new set of energy related independent variables $\omega$ and $\upsilon$, for forward and backward characteristics, respectively. These energy related independent variables allow us to apply the balance laws in \eqref{mbl}, to prove the uniqueness of characteristic (step 1).
In this step, one main idea is to use some weighted Riemannian distance to measure the distance between $\omega_1$ and $\omega_2$ (also for $\upsilon_1$ and $\upsilon_2$) corresponding to two different characteristics, and prove some Lipschitz property of this distance, which directs to the uniqueness of characteristic. After proving the uniqueness of characteristic, for any given solution $u(t,x)$, we show that it satisfies a semi-linear system on some dependent variables. Since this semi-linear system always has a unique solution, we prove our uniqueness result (step 2). This step is essentially a reverse process of the existence proof in \cite{H}. 

\section{The existence and uniqueness of characteristics \label{step1}}

In this section, we prove the existence and uniqueness of characteristics, which play a crucial role in our analysis. Motivated by a recent paper \cite{BCZ}, the key idea is
simply to write a pair of ODEs along forward and backward characteristic curves starting at a given point $\bar{y}$, respectively. If each of these two equations admits a unique solution for a.e. $\bar{y}$, then all characteristic curves can be uniquely determined. Note, as mentioned in Section \ref{sec_mainresult}, for any fixed $\bar{y}\in\mathbb{R}$, the Cauchy problems
\bel{dxpm}\dot x^-(t)~=~\lambda_-\big(x^-(t),u(t,x^-(t))\big),\qquad\qquad \dot x^+(t)~=~\lambda_+\big(x^+(t),u(t,x^+(t))\big),
\eeq
with initial data
\bel{icy}
x^-(0)~=~\bar y,\qquad\qquad x^+(0)~=~\bar y,\eeq
might have multiple solutions, since $u(t,x)$ is only H\"older continuous. Here the upper dot denote a derivative w.r.t. time. To overcome this difficulty, our analysis bases on two key points.

\begin{itemize}
  \item[$\blacklozenge.$]  We introduce a pair of variables corresponding to the forward and backward energies, related to the original Eulerian coordinates $(t,x)$ by the following transformation
\bel{xbeta} x(t,\omega) + \int_{-\infty}^{x(t,\omega)} \Tilde{R}^2(t,\xi)\, d\xi~=~\omega\,,\eeq
\bel{ybeta} y(t,\upsilon) + \int_{-\infty}^{y(t,\upsilon)} \Tilde{S}^2(t,\xi)\, d\xi~=~\upsilon\,.\eeq
Here $w$ and $\nu$ denote an energy related parameter of the backward characteristic and the forward characteristic, respectively.
Such energy variables help us select the ``correct" characteristic after the collapse
of characteristics at the time of energy concentration (or wave breaking).
  \item[$\blacklozenge.$] $u=u(t, x)$ is a conservative solution of the Cauchy problem \eqref{vwl}--\eqref{ID} with the balance laws \eqref{balance}. Thus, the characteristic curves $t\mapsto x^{\pm}(t)$ satisfy the additional equations
\bel{ceq1}
{d\over dt}  \int_{-\infty}^{x^-(t)}\Tilde{R}^2(t,x)\, dx~=~\int_{-\infty}^{x^-(t)} \big(\frac{2c_2a_1}{\alpha(c_2-c_1)}R^2 S-\frac{2c_1a_2}{\alpha(c_2-c_1)}RS^2-\frac{2c_1c_2b}{\alpha(c_2-c_1)}RS\big)
\, dx\,,\eeq
\bel{ceq2}{d\over dt}  \int_{-\infty}^{x^+(t)}\Tilde{S}^2(t,x)\, dx~=~-\int_{-\infty}^{x^+(t)} \big(\frac{2c_2a_1}{\alpha(c_2-c_1)}R^2 S-\frac{2c_1a_2}{\alpha(c_2-c_1)}RS^2-\frac{2c_1c_2b}{\alpha(c_2-c_1)}RS\big)
\, dx\,.\eeq
By these two equations together with all equations in \eqref{dxpm}--\eqref{icy}, we will eventually obtain that the characteristic curves can be uniquely determined.
\end{itemize}

Let $u=u(t,x)$ be a conservative solution of (\ref{vwl}), which satisfies all the properties listed in Theorems \ref{thm_ex} and \ref{thm_ec}.
As mentioned above, it is convenient to work with
an adapted set of variables $x(t, \omega)$, $y(t,\upsilon)$, instead of the variables $(t,x)$ by the integral relations \eqref{xbeta}--\eqref{ybeta}. At times $t$ where the measures $\mu_-^t,\mu_+^t$ are not absolutely continuous w.r.t. Lebesgue measure, we can define the points $x(t,\omega)$ and
$ y(t,\upsilon)$ by setting
\begin{equation*}\label{xadef}
x(t,\omega)~:=~\sup\Big\{ x\,;~~x+\mu^t_-\bigl(\,(-\infty,x]\,
\bigr) ~<~\omega\Big\}\,,\end{equation*}
\begin{equation*}\label{ybdef}y(t,\upsilon)~:=~\sup\Big\{ x\,;~~x+\mu^t_+\bigl(\,(-\infty,x]\,\bigr)
 ~<~\upsilon\Big\}\,,\end{equation*}
for $\omega,\upsilon\in \mathbb R$. Hence,
\bel{xa}
\omega~=~x(t,\omega)+\mu_-^t\Big(\bigl(-\infty\,, ~x(t,\omega)
\bigr)\Big)+\theta\cdot \mu^t_- \Big(\bigl\{x(t,\omega)
\bigr\}\Big)\,,\eeq
\bel{yb}
\upsilon~=~y(t,\upsilon)+\mu_+^t\Big(\bigl(-\infty\,,~ y(t,\upsilon)
\bigr)\Big) +\theta'\cdot \mu^t_- \Big(\bigl\{y(t,\upsilon)
\bigr\}\Big)\,,\eeq
for some $\theta,\theta'\in [0,1]$. Since the measures $\mu^t_-$, $\mu^t_+$ are both positive and bounded, it is clear that these points are well defined. Notice that the above definitions coincides with \eqref{xbeta}--\eqref{ybeta} at any time $t$ where the measures $\mu_-^t,\mu_+^t$ are absolutely continuous w.r.t. the Lebesgue measure.


We give the first lemma of this
section which is helpful to establish the property of $x,y$ and $u$ as functions
of the variables $t, \omega,\upsilon$.

\begin{Lemma} \label{xy_lem}
For every fixed $t$,  the maps $\omega\mapsto x(t,\omega)$  and
 $\upsilon\mapsto y(t,\upsilon)$  are both Lipschitz continuous with constant 1.
 Moreover, for fixed $\omega,\upsilon$, the maps $t\mapsto  x(t,\omega)$ and
 $t\mapsto  y(t,\upsilon)$
are absolutely continuous, locally H\"older continuous with exponent $1/2$, and have locally bounded variation.
\end{Lemma}

\begin{proof}[\bf Proof] Let's first make some calculations which will be used throughout this paper. Applying \eqref{con} and \eqref{lambda}, one has
\begin{equation}\label{not1}
\frac{c_2-c_1}{\alpha}=\lambda_+-\lambda_-=\frac{2\sqrt{\beta^2+\alpha^2\gamma^2}}{\alpha^2}\geq \frac{2\gamma_1}{\alpha_2},
\end{equation}
and
\begin{equation}\label{not2}
\frac{1}{\underline{M}}:=\frac{\alpha_1^2\gamma_1^2}{2(\beta_2^2+\alpha^2_2\gamma_2^2+\beta_2\sqrt{\beta^2_2+\alpha^2_2\gamma_2^2})}
\leq\big|\frac{c_i}{c_2-c_1}\big|\leq \frac{\beta_2+\sqrt{\beta^2_2+\alpha^2_2\gamma_2^2}}{2\alpha_1\gamma_1}=:\overline{M}
\end{equation}
for $i=1,2$. Now we prove this lemma by four steps.

{\bf Step 1.} For the first statement, observe that if
$$x_1~:=~ x(t,\omega_1)~<~x(t,\omega_2)~=:~x_2\,,$$
then a direct computation implies that
$$x_2-x_1 +\mu_-^t\bigl(\,(x_1, x_2)\,\bigr)~\leq~\omega_2-\omega_1\,.$$
This gives that
$$x_2-x_1 ~\leq~\omega_2-\omega_1\,,$$
proving the Lipschitz continuity of the map $\omega\mapsto x(t,\omega)$.
Of course, the same argument is valid for the map $\upsilon\mapsto y(t,\upsilon)$.

{\bf Step 2.} We continue with the second statement.
Denote $\mu_-^t\otimes\mu_+^t$ be
 the product measure on $\mathbb{R}^2$ and
define the wave interaction potential as
\begin{equation*}
Q(t)~:=~\Big(\mu_-^t\otimes\mu_+^t\Big)\Big( \{(x,y)\,;~~x>y\}\Big).\end{equation*}
Using
the balance laws  (\ref{mbl}) and \eqref{not1}, then recalling that $\Tilde{R}^2(t,\cdot)$ and $\Tilde{S}^2(t,\cdot)$ provide the
absolutely continuous parts of $\mu_-^t$ and $\mu_+^t$, respectively. It holds that
\begin{equation*} \begin{split}
{d\over dt}Q(t)&\leq~-\int_{-\infty}^{+\infty}\frac{c_2-c_1}{\alpha}
 \Tilde{R}^2(t,x)\Tilde{S}^2(t,x)\, dx +
\int\int^{+\infty}_y|G(t,x)|\, dx\,d\mu_+^t(y)\\
&\qquad+
\int \int^x_{-\infty}|G(t,y)|\, dy\,d\mu_-^t(x) \\
&\leq~-\frac{2\gamma_1}{\alpha_2}\int_{-\infty}^{+\infty} \Tilde{R}^2(t,x)\Tilde{S}^2(t,x)\, dx +
 \bigr(\mu_-^t(\mathbb R)+\mu_+^t(\mathbb R)\bigl) \int_{-\infty}^{+\infty}|G(t,x)|\, dx\\
&\leq~-\frac{2\gamma_1}{\alpha_2}\int_{-\infty}^{+\infty} \Tilde{R}^2\Tilde{S}^2\, dx  +
\widehat{C} E_0 \int_{-\infty}^{+\infty}\Big(|\Tilde{R}^2S|+|R\Tilde{S}^2|+\Tilde{R}^2+\Tilde{S}^2\Big)\, dx,
\end{split}
\end{equation*}
here we have use the fact that \begin{equation}\label{Gest}
|G|\leq \widehat{C}\big(|\Tilde{R}^2S|+|R\Tilde{S}^2|+\Tilde{R}^2+\Tilde{S}^2\big),
\end{equation}
for some positive constant $\widehat{C}$. Moreover, for each $\epsilon>0$, we notice by \eqref{not2} that
\begin{equation}\label{rsest}
|S|\leq \frac{1}{2\sqrt{\epsilon}}+\frac{\sqrt{\epsilon}}{2}\underline{M}\Tilde{S}^2\quad {\rm and}\quad |R|\leq \frac{1}{2\sqrt{\epsilon}}+\frac{\sqrt{\epsilon}}{2}\underline{M}\Tilde{R}^2.\end{equation}
We then choose $\epsilon>0$ such that $$\sqrt{\epsilon}\leq \frac{\gamma_1}{\alpha_2\underline{M}\widehat{C}E_0}$$ to get
\bel{DQ} \begin{split}
{d\over dt}Q(t)&\leq~-\frac{\gamma_1}{\alpha_2}\int_{-\infty}^{+\infty} \Tilde{R}^2\Tilde{S}^2\, dx +
\widehat{C} E_0(1+\frac{\alpha_2\underline{M}\widehat{C} E_0}{2\gamma_1}) \int_{-\infty}^{+\infty}\big(\Tilde{R}^2+\Tilde{S}^2\big)\, dx\\
&\leq~-\frac{\gamma_1}{\alpha_2}\int_{-\infty}^{+\infty} \Tilde{R}^2\Tilde{S}^2\, dx +
\widehat{C} E_0^2(1+\frac{\alpha_2\underline{M}\widehat{C} E_0}{2\gamma_1}).
\end{split}
\eeq

Since $Q(t)\leq E_0^2$ for every time $t$, from (\ref{DQ}) one has
\bel{SRb}\begin{split}
\int_0^T \int_{-\infty}^{+\infty} \Tilde{R}^2(t,x)\Tilde{S}^2(t,x)\, dx\, dt&\leq~
{\alpha_2\over \gamma_1}\left[ Q(0)-Q(T) + \widehat{C} E_0^2T(1+\frac{\alpha_2\underline{M}\widehat{C} E_0}{2\gamma_1})\right]\\
&\leq~
\frac{2\alpha_2E_0^2}{\gamma_1}+\frac{\alpha_2}{\gamma_1}\widehat{C} E_0^2T(1+\frac{\alpha_2\underline{M}\widehat{C} E_0}{2\gamma_1})\,.\end{split}\eeq

{\bf Step 3.} For a given $\tau$ and any $\ve\in\,(0,1]$, by using \eqref{rsest} and (\ref{SRb}), we have
\begin{equation*}
\begin{split}
&\quad\int_\tau^{\tau+\ve} \int_{-\infty}^{+\infty} |G(t,x)|
\,dx\, dt\leq ~\widehat{C}\int_\tau^{\tau+\ve} \int_{-\infty}^{+\infty} \Big(|\Tilde{R}^2S|+|R\Tilde{S}^2|+\Tilde{R}^2+\Tilde{S}^2\Big)\,dx\, dt\\
&\ds\leq~ \widehat{C}\int_\tau^{\tau+\ve} \int_{-\infty}^{+\infty} \Big(\frac{1}{2\sqrt{\varepsilon}}(\Tilde{R}^2+\Tilde{S}^2)+\varepsilon^{\frac{1}{2}}\Tilde{R}^2\Tilde{S}^2\underline{M}
+\Tilde{R}^2+\Tilde{S}^2\Big)\,dx\, dt\\
&\leq~ \frac{\widehat{C}}{2}\varepsilon^{\frac{1}{2}}E_0+\widehat{C}\varepsilon^{\frac{1}{2}}\underline{M}\Big[\frac{2\alpha_2E_0^2}{\gamma_1}+\frac{\alpha_2}{\gamma_1}\widehat{C} E_0^2\varepsilon(1+\frac{\alpha_2\underline{M}\widehat{C} E_0}{2\gamma_1})\Big]+\widehat{C}\varepsilon E_0\\
& \leq~\widehat{C} E_0\left[\frac{3}{2}+\frac{2\alpha_2E_0\underline{M}}{\gamma_1}+\frac{\alpha_2\widehat{C}E_0\underline{M}}{\gamma_1} +\frac{\alpha_2^2\underline{M}^2\widehat{C}^2 E_0^2}{2\gamma_1^2}
\right]\ve^{1/2} .
 \end{split}
 \end{equation*}
 Thus, the function $\varsigma$ defined by
 \bel{zdef}
 \begin{split}
 \varsigma(\tau)~&:=~\int_0^\tau \int_{-\infty}^{+\infty}
|G(t,x)|\,dx\, dt\\
&=~\int_0^\tau \int_{-\infty}^{+\infty}
 \left|\frac{2c_2a_1}{\alpha(c_2-c_1)}R^2 S-\frac{2c_1a_2}{\alpha(c_2-c_1)}RS^2-\frac{2c_1c_2b}{\alpha(c_2-c_1)}RS\right|
\,dx\, dt\end{split}\eeq
is locally H\"older continuous, nondecreasing, with sub-linear growth.
Since $G(t,x)\in \L^1([0,T]\times\mathbb{R})$, by Fubini's theorem the map
$t\mapsto \int
 |G(t,x)|\,dx$ is  in $\L^1(\mathbb{R})$.   By its definition (\ref{zdef}),
the function $\varsigma$ is absolutely continuous.
Moreover, for $0<t_2-t_1\leq 1$ we have
$$
\varsigma(t_2)-\varsigma(t_1)~\leq ~C_1(t_1-t_2)^{1/2},
$$
where the constant $C_1$ is defined as
\begin{equation*}\label{const1}
C_1~:=~ \widehat{C} E_0\left[\frac{3}{2}+\frac{2\alpha_2E_0\underline{M}}{\gamma_1}+\frac{\alpha_2\widehat{C}E_0\underline{M}}{\gamma_1} +\frac{\alpha_2^2\underline{M}^2\widehat{C}^2 E_0^2}{2\gamma_1^2}
\right].
\end{equation*}

{\bf Step 4.}
In view of \eqref{con} and \eqref{lambda}, there exist constants $\underline{N}$ and $\overline{N}$ such that
\begin{equation*}\label{lambdaest}
\underline{N}:=\frac{\gamma_1^2 }{\beta_2+\sqrt{\beta^2_2+\alpha^2_2\gamma_2^2}}\leq
|\lambda_{\pm}|\leq\frac{\beta_2+\sqrt{\beta^2_2+\alpha^2_2\gamma_2^2}}{\alpha_1^2}=:\overline{N}.
\end{equation*}
For any $t_1<t_2$ and any $\omega$, since the family of measures
$\mu_-^t$ satisfies the balance law (\ref{mbl}) with velocity $\lambda_-$, we obtain
\begin{equation*}
\mu_-^{t_2}\Big( \big(-\infty\,,~ x(t_1,\omega)\big)\Big)
~\geq~\mu_-^{t_1}\Big( \bigl(-\infty\,,~ x(t_1,\omega)\big)\Big) -
\bigl[\varsigma(t_2)
-\varsigma(t_1)\bigr]\,,\end{equation*}
\begin{equation*}
\mu_-^{t_2}\Big( \big(-\infty\,,~ x(t_1,\omega)-\overline{N}(t_2-t_1)\big)\Big)
~\leq~\mu_-^{t_1}\Big( \big(-\infty\,,~ x(t_1,\omega)\big)\Big) +\big[\varsigma(t_2)
-\varsigma(t_1)\big]\,.\end{equation*}
This follows from the definition (\ref{xa}) that
\begin{equation*}
x(t_1,\omega)-\overline{N}(t_2-t_1) -  \bigl[\varsigma(t_2)
-\varsigma(t_1)\bigr]  ~\leq~ x(t_2,\omega) ~\leq ~x(t_1,\omega)+\bigl[\varsigma(t_2)
-\varsigma(t_1)\bigr].\end{equation*}
By the properties of the function $\varsigma$, proved in step {\bf 3}, we complete the proof for the map $t\mapsto x(t,\omega)$.
Of course, the same argument is valid for the map $t\mapsto y(t,\upsilon)$.
\end{proof}
\begin{Remark}
The estimates in the proof of this Lemma, particularly \eqref{SRb}, are crucial for the Lipschitz metric result in \cite{CCS}.
\end{Remark}

\vskip 0.3cm

Now, we show in the next lemma that for a conservative solution the characteristics can be uniquely determined by combining the characteristic equations (\ref{dxpm})--(\ref{icy}) and the  balance laws (\ref{ceq1})--(\ref{ceq2}). One main spirit in the proof is to introduce a weighted distance, including some wave interaction potentials, in order to control the possible increase of forward or backward energy during wave interactions.

More precisely, we will show that there exist a unique pair of  backward and forward characteristics $t\mapsto x^-(t)=x(t,\omega(t))$ and $t\mapsto x^+(t)=y(t,\upsilon(t))$, starting from the point $(0,\bar y)$,  satisfying equations (\ref{dxpm})--(\ref{icy}), as well as (\ref{ceq1})--(\ref{ceq2}), with $t\mapsto \omega(t)$ and $t\mapsto \upsilon(t)$ be the unique solution to the following integral equations
\begin{equation}\label{uweqn}
\begin{split}
\omega(t)=&\bar{\omega}+\int_0^t \Big\{\frac{c_1}{\alpha}\big(x(s, \omega(s)),u(s,x(s, \omega(s))\big)\\
&+\int_{-\infty}^{x(s, \omega(s))}
\big[\frac{2c_2a_1}{\alpha(c_2-c_1)}R^2 S-\frac{2c_1a_2}{\alpha(c_2-c_1)}RS^2-\frac{2c_1c_2b}{\alpha(c_2-c_1)}RS
\, dx\big]\Big\}\,ds,\\
 {\rm and}~\upsilon(t)=&\bar\upsilon +\int_0^t \Big\{\frac{c_2}{\alpha}\big(y(s, \upsilon(s)),u(s,y(s, \upsilon(s))\big)\\
&-\int_{-\infty}^{y(s, \upsilon(s))}
\big[\frac{2c_2a_1}{\alpha(c_2-c_1)}R^2 S-\frac{2c_1a_2}{\alpha(c_2-c_1)}RS^2-\frac{2c_1c_2b}{\alpha(c_2-c_1)}RS
\, dx\big]\Big\}\,ds,
\end{split}\end{equation}
respectively. Here \[\bar{\omega}=\omega(0)=\bar{y}+\int^{\bar{y}}_{-\infty}\Tilde{R}^2(0,x)\,dx,\]
 and \[\bar{\upsilon}=\upsilon(0)=\bar{y}+\int^{\bar{y}}_{-\infty}\Tilde{S}^2(0,x)\,dx.\] Now, we state this fundamental result.

\begin{Lemma} \label{un_lem}
Let $u=u(t,x)$ be a conservative solution of (\ref{vwl})--\eqref{ID}.
Then, for any $\bar y\in \mathbb{R}$, there exist unique Lipschitz continuous maps
$t\mapsto x^\pm(t)$ which satisfy (\ref{dxpm})--(\ref{icy})
together with (\ref{ceq1})--(\ref{ceq2}).
\end{Lemma}

\begin{proof}[\bf Proof]  We claim  that there exists a unique function
$t\mapsto \omega(t)$ such that
\begin{equation*}\label{xal}x^-(t)~=~x(t,\omega(t))\end{equation*}
satisfies the first equation in (\ref{dxpm})--(\ref{icy}) and (\ref{ceq1}).
Of course, the same argument can be applied to the map $t\mapsto x^+(t)$ which satisfy the second equation in \eqref{dxpm}--\eqref{icy} and \eqref{ceq2}. Now we prove this claim by five steps.

{\bf Step 1.} We first construct the equation for $\omega(t)$ in the following way.
Summing the first equation in (\ref{dxpm}) with
(\ref{ceq1}) and integrating the resultant equation w.r.t.~time, one has an integral equation for $\omega$,
\bel{ceq3}\bega{l}\ds
\omega(t)=~x^-(t)+\mu_-^{t}\Big( \big(-\infty\,,~ x^-(t)\big)\Big)+
\theta \cdot\mu_-^{t}\Big( \{x^-(t)\}\Big)
\cr\cr\ds\qquad=~\bar\omega + \int_0^t \left( \lambda_-(x^-(s),u(s, x^-(s)))+\int_{-\infty}^{x^-(s)} G(s,x)
\, dx\right) ds\,,\enda\eeq
for some $\theta\in [0,1]$ and $G$ defined in \eqref{balance}.
Here
\bel{ica}\bar \omega~=~\omega(0)~=~\bar y +\int_{-\infty}^{\bar y} \Tilde {R}(0,x)\, dx\,.\eeq
We further observe that the equation (\ref{ceq3}) is equivalent to
\bel{ceq5}
\dot\omega(t)~=~ \lambda_-\big(x(t, \omega(t)),u(t,x(t, \omega(t)))\big)
+\int_{-\infty}^{x(t, \omega(t))}
G(t,x)
\, dx\,,\eeq
with initial data (\ref{ica}). To complete the proof we need to prove
that the integral equation  (\ref{ceq3}) has a unique solution
$\omega(t)$.
Moreover, the function $t\mapsto x^-(t) = x(t,\omega(t))$  satisfies the first equation in (\ref{dxpm}), as well as
(\ref{ceq1}).
\vskip 0.2cm
{\bf Step 2.} In this step, we show the existence of a solution to the integral equation
(\ref{ceq3}). Let's begin on the interval $t\in [0,1]$, then iterate the argument by induction.
More precisely, consider a set $\K$ of H\"older continuous functions by
\begin{equation*}
\ds \K~:=~\{f\in \C^{1/2}([0,1])\,;~~~\|f\|_{\C^{1/2}}~\leq~ C_K,
~~~ f(0)=\bar\omega\},
\end{equation*}
for a suitable constant $C_K$. On this set we claim that the Picard map $\P: \C^0([0,1])\mapsto \C^0([0,1])$, defined as
\begin{equation*}
\P\omega(t)~:=~\bar\omega + \int_0^t \left( \lambda_-\big(x(s, \omega(s)),u(s, x(s, \omega(s)))\big)+\int_{-\infty}^{x(s, \omega(s))} G(s,x)
\, dx\right) ds\,.
\end{equation*}
 a continuous transformation of
the compact convex  set $\K\subset \C^0([0,1])$ into itself.
 We omit the detailed proof here for brevity,
since a similar argument can be found in \cite{BCZ}.
By Schauder's fixed point theorem, we derive that the integral equation  (\ref{ceq3})
has at least one solution.  Iterating the argument, this solution can be
extended to any time interval $t\in [0,T]$.
\vskip 0.2cm
{\bf Step 3.} Now we are in a position to prove the uniqueness of the solution to
(\ref{ica})--(\ref{ceq5}) by controlling the highest order terms in \eqref{ceq5}. We consider the weight
\begin{equation*}\label{weight} W(t,\omega)~:=e^{\kappa A^+(t,\omega)}
\end{equation*}
with
\begin{equation*}\label{A+}\qquad\qquad A^+(t,\omega)~:=~\mu_+^t\Big((-\infty,~x(t,\omega)]\Big)+
[\varsigma(T)-\varsigma(t)].
\end{equation*}
Here $\varsigma$ is the function defined at (\ref{zdef}), while
\begin{equation*}\label{mudef}\kappa~:=~\frac{\alpha_2}{2\gamma_1}(\widehat{C}\sqrt{\underline{M}}
+\widehat{C}+\frac{\overline{C}\sqrt{\underline{M}}}{4\gamma_1})
\,,\end{equation*}
 where $\underline{M}$ and $\widehat{C}$ are defined in \eqref{not2} and \eqref{Gest}, respectively. and we assume $|\frac{\partial\lambda_-}{\partial x}|,|\frac{\partial\lambda_-}{\partial u}|\leq \overline{C}$ for some constant $\overline{C}$.

  We recall that $\varsigma(T)-\varsigma(t)$
 provides an upper bound on the energy transferred from backward to forward
 moving waves and conversely, during the time interval $[t,T]$.
In turn, $A^+(t,\omega)$ yields an upper bound on the total
energy of forward moving waves that can cross the backward characteristic
$x(\cdot,\omega)$
during the time interval  $[t,T]$.
For any $\omega_1<\omega_2$ and $t\geq 0$, we further define a weighted distance as
\begin{equation*}\label{RD}
d^{(t)}(\omega_1,\omega_2)~:=~\int_{\omega_1}^{\omega_2}
W(t,\omega)\, d\omega.
\end{equation*}
By the Gronwall's lemma, we can prove that
\begin{equation*}\label{Gro}
d^{(t)}\bigl(\omega_1(t),\omega_2(t)\bigr)~
\leq~ e^{C_0t}\, d^{(0)}\bigl(\omega_1(0),\omega_2(0)
\bigr),
\end{equation*}
with $C_0:=(\frac{\widehat{C}}{2}
+\frac{\overline{C}}{2\gamma_1})\sqrt{\underline{M}}+\widehat{C}+\overline{C}$. We omit the proof here for brevity, since an entirely similar approach of this result can be found in \cite{BCZ}. Thus,
for every initial value $\bar\omega$, the solution
of (\ref{ica})--(\ref{ceq5}) is unique.
\vskip 0.2cm
{\bf Step 4.}  Now we come to show that  $x^-(t)=x(t,\omega(t))$
satisfies the first equation in (\ref{dxpm}) at almost every time.
Applying the classical theorem of Lebesgue, we derive that
$$
\lim_{r\to 0^+}~{1\over \pi r^2}\,\dint_{(\tau-t)^2+(y-x)^2\leq r^2}
G(\tau,y)dyd\tau~=~G(t,x),
$$
for all $(t,x)\in ~(0,T)\,\times \mathbb{R}$ outside a null
set $\N_2$ whose 2-dimensional measure is zero, since $G(t,x)\in \L^1([0,T]\times\mathbb{R})$. In particular, if one divides by $r$ instead of $r^2$, by Corollary 3.2.3 in  \cite{Zi}
there is a set $\N_1\subset \N_2$
whose 1-dimensional Hausdorff measure is zero and  for every $(t,x)\not\in \N_1$
 $$
\limsup_{r\to 0^+}~{1\over r}\,\dint_{(\tau-t)^2+(y-x)^2\leq r^2}
G(\tau,y)dyd\tau~=~0.
$$

Moreover, by the definition of absolutely continuous and the fact that the map $\omega\mapsto x(t,\omega)$ is contractive, we can prove the map $t\mapsto
x^-(t)= x(t,\omega(t))$ is absolutely continuous. The details can be found in \cite{BCZ}.

Thus, there exists a null 1-dimensional set $\N\subset [0,T]$
such that
\begi
\item[(i)]  For every $\tau\notin\N$ and $x\in \mathbb{R}$ one has $(\tau,x)\notin \N_1$\,;
\item[(ii)] If $\tau\notin \N$ then the map $t\mapsto \zeta(t)$ in (\ref{zdef})
is differentiable at $t=\tau$. Moreover, $\tau$
is a Lebesgue point of the derivative $\varsigma'$;
\item[(iii)] The functions $t\mapsto x^-(t)$ and
$t\mapsto\omega(t)$ are differentiable at each point   $\tau\in [0,T]\setminus\N$.
Moreover, each point $\tau \notin \N$ is a Lebesgue point of the derivatives $\dot x^-$ and $\dot\omega$.\endi

Let $\tau\not\in \N$.
 We claim  that the map $t\mapsto x^-(t)=x(t,\omega(t))$
satisfies the first equation in (\ref{dxpm}) at time $t=\tau$. Assume, on the contrary, that $\dot x^-(\tau)\not= \lambda_-\big(x^-(\tau),u(\tau, x^-(\tau))\big)$.
Without loss of generality, let
\begin{equation*}\label{ass}\dot x^-(\tau)~=~ \lambda_-\big(x^-(\tau),u(\tau, x^-(\tau))\big)+2\ve_0\end{equation*}
 for some $\ve_0>0$.   The case $\ve_0 <0$ is entirely similar.
To derive a contradiction we observe that, for all $t\in (\tau, \tau+\delta]$, with $\delta>0$ small enough one has
\bel{xpmt}X(t)~:=~
x^-(\tau) +(t-\tau) [\lambda_-\big(x^-(\tau),u(\tau, x^-(\tau))\big)+\ve_0 ]~<~x^-(t).\eeq
We also observe that if  $\vp$  is Lipschitz continuous with compact support then the identity in (\ref{mbl}) is still true.

For any $\epsilon>0$ small, we still use the Lipschitz function with compact support constructed in \cite{BCZ} as the test function, that is
 \begin{equation*}
 \varphi^{\epsilon}(s,y)~:=~\min\{ \varrho^{\epsilon}(s,y),
\,\chi^\epsilon(s)\},\end{equation*}
where
$$\rho^{\epsilon}(s,y)~:=q~\left\{\bega{cl} 0 \qquad &\hbox{if}\quad y
\leq -\epsilon^{-1},\cr
\epsilon^{-1}(y+\epsilon^{-1}) \qquad &\hbox{if}\quad -\epsilon^{-1}\leq y\leq
\epsilon-\epsilon^{-1},\cr
1\qquad &\hbox{if}\quad \epsilon-\epsilon^{-1} \leq y\leq X(s),\cr
1-\epsilon^{-1}(y-X(s))\qquad &\hbox{if}\quad  X(s)\leq y\leq X(s)+\epsilon,\cr
0\qquad &\hbox{if}\quad y\geq X(s)+\epsilon,\enda\right.$$
\begin{equation*}\label{timtest}\chi^\epsilon(s)~:=~\left\{\bega{cl} 0\qquad
&\hbox{if}\quad s\leq \tau-\epsilon,\cr
\epsilon^{-1}(s-\tau+\epsilon)\qquad &\hbox{if}\quad \tau-\epsilon\leq s\leq \tau,\cr
1\qquad &\hbox{if}\quad \tau\leq s\leq  t,\cr
1-\epsilon^{-1}(s-t) \qquad &\hbox{if}\quad t\leq s<t+\epsilon,\cr
0 \qquad &\hbox{if}\quad s\geq t+\epsilon.\enda\right.
\end{equation*}
Using $\vp^{\epsilon}$ as test function in the first equation in  (\ref{mbl}), we get
\bel{vpe}
\int \Big[\int (\vp^{\epsilon}_t+\lambda_-\vp^{\epsilon}_x)\,d\mu_-^t
+ \int G(t,x)\, \vp^{\epsilon}
 \, dx \Big]\,dt~=~0.
\eeq
Suppose $t$ is sufficiently close to $\tau$, then
for $s\in [\tau, \, t]$ and $x$ close to $x^-(\tau)$, we have
$$0 ~=~ \vp^{\epsilon}_t + [\lambda_-\big(x(\tau),u(\tau, x(\tau))\big)+\ve_0 ] \vp^{\epsilon}_x ~
\leq~\vp^{\epsilon}_t +\lambda_-\big(x, u(s,x)\big) \vp^{\epsilon}_x\,,$$
because  $\lambda_-\big(x,u(s,x)\big)<\lambda_-\big(x,u(\tau, x(\tau))\big)+\ve_0 $ and $\vp^{\epsilon}_x\leq 0$.
Since the family of measures $\mu_{-}^{t}$ depends continuously on $t$
in the topology of weak convergence, taking the limit of (\ref{vpe}) as
 $\epsilon\to 0$, for $\tau,t\notin\N$ we obtain
\bel{55}\bega{rl} 0&
\geq~
\ds\mu_-^\tau\Big((-\infty,~x^-(\tau)]\Big)-\mu_-^t\Big((-\infty,~X(t)]\Big)+ \int_\tau^t\int_{-\infty}^{X(s)}G(s,x)\,dxds.
\enda\eeq
In turn,  for $t\in (\tau,\tau+\delta]$, \eqref{xpmt} and (\ref{55})  implies
\bel{66}\begin{split} &\quad\mu_-^t\Big((-\infty,~x^-(t))\Big)~\geq~\mu_-^t\Big((-\infty,~X(t)]\Big)\\
&\geq~\ds\mu_-^\tau\Big((-\infty,~x^-(\tau)]\Big)
\ds+\int_\tau^t\int_{-\infty}^{x^-(s)}G(s,x)\,dxds+ o(t-\tau),\end{split}
\eeq
where $$ o(t-\tau)~:=~-\int_\tau^t \int_{X(s)}^{x^-(s)}G(s,y)\,dy ds,
$$
since $\tau\notin\N$, the last term is a higher order infinitesimal, satisfies
$\lim\limits_{t\to\tau}\frac{o(t-\tau)}{t-\tau}~=~0.$
Thus, for $t$ sufficiently to $\tau$, it follows from (\ref{66}) that
\bel{ccc}\begin{split}&
\omega(t) - \omega(\tau) ~\geq~
\bigg[x^-(t) + \mu_-^t\Big( (-\infty,\, x^-(t))\Big)\bigg] - \bigg[ x^-(\tau) +
 \mu_-^\tau\Big( (-\infty,\, x^-(\tau)]\Big)\bigg]\\
&
\qquad \geq~ \Big[ \lambda_-\bigl(x^-(\tau),u(\tau, x^-(\tau))\bigr)+\ve_0\Big](t-\tau) +
\int_\tau^t \int_{-\infty}^{x^-(s)}G(s,y)\,dy ds + o(t-\tau).
\end{split}
\eeq
Differentiating (\ref{ccc}) w.r.t.~$t$ at $t=\tau$, we find
$$\dot \omega(\tau) ~\geq~ \Big[ \lambda_-\bigl(x^-(\tau),u(\tau, x^-(\tau))\bigr)+\ve_0\Big]+
\int_{-\infty}^{x^-(\tau)}G(s,y)\,dy ds,$$
which is a contradiction
with (\ref{ceq5}). As a consequence, the first equation in  (\ref{dxpm}) must hold.
\vskip 0.2cm

\vskip 0.2cm
{\bf Step 5.}
Finally, we prove the uniqueness of $x^-(t)$.
Assume there are two different solution
$x^-_1(t)$ and $x^-_2(t)$ with $x^-_1(0)=x^-_2(0) = \bar y$, both satisfying the first equation in (\ref{dxpm}) together with (\ref{ceq1}).
Consider two functions $\omega_1(t)$ and $\omega_2(t)$ as
$$\omega_i(t)~=~x_i^-(t) + \int_{-\infty}^{\bar y}\Tilde{R}^2(0,x)\, dx
+\int_0^t\int_{-\infty}^{x_i^-(t)}G(s,x)\, dx\, ds\,,$$
for $i=1,2$. Then $x_i^-(t) = x(t,\omega_i(t))$, moreover
$\omega_1$ and $\omega_2$ satisfy (\ref{ceq3})
with the same initial data
$$\omega_1(0) ~=~ \omega_2(0)~=~\bar y + \int_{-\infty}^{\bar y}\Tilde{R}^2(0,x)\, dx\,.$$
The uniqueness result of $\omega$ proved in step {\bf 3} now implies~
$x_1^-(t) = x(t,\omega_1(t))=x(t,\omega_2(t))=x_2^-(t)$.
\end{proof}

\section{An equivalent semi-linear system}\label{sub_3.5}

In this section, we construct a semilinear system under characteristic independent variables $X,Y$, which are corresponding to $\omega(0)$ and $\upsilon(0)$ and taking constant values along forward and backward characteristics,  respectively.  We need to define some dependent variables in the $(X,Y)$ coordinates, and show that these variables satisfy a semi-linear system whose solution is unique. In turn, this provides a direct proof of the uniqueness of conservative solutions to
(\ref{vwl})--(\ref{ID}) with general initial data  $u_0\in H^1({\mathbb{R}}),\  u_1\in L^2({\mathbb{R}})$.

 Let $u=u(t,x)$ be a conservative solution of (\ref{vwl})--\eqref{ID}. For any couple $(X,Y)\in\mathbb{R}^2$, a unique point $(t,x)$ can be determined as follows.
Choose points $\bar x=x_0(X)$ and $\bar y=y_0(Y)$ such that
\begin{equation*}\label{xy}
X~=~\bar x +\int_{-\infty}^{\bar x}\Tilde{R}^2(0,x)dx,\qquad
\qquad Y~=~\bar y +\int_{-\infty}^{\bar y}\Tilde{S}^2(0,x)dx.\end{equation*}
Using Lemma \ref{un_lem}, there exists a unique backward characteristic $t\mapsto x^-(t,\bar x)$
starting at $\bar x$, and a unique forward characteristic $t\mapsto x^+(t,\bar y)$
starting at $\bar y$.
Without loss of generality, we assume that $\bar x\geq\bar y$, then define $(t(X,Y), x(X,Y))$ be the unique point
where these two characteristics cross, namely
\begin{equation*}\label{xtXY}x^-(t(X,Y),\,\bar x) ~=~x^+(t(X,Y),\,\bar y)~=~ x(X,Y),\end{equation*}
and the function $u(X,Y)$ is defined by
\begin{equation*}\label{uXY}
u(X,Y)~:=~u\bigl(t(X,Y), x(X,Y)\bigr).\end{equation*}
From the above definitions, we can state the following lemma, using a similar method as in \cite{BCZ}, we omit the proof here for brevity.
\begin{Lemma}\label{nlip_lem}
The map $(X,Y)\to (t,x,u)(X,Y)$ is  locally Lipschitz continuous.\end{Lemma}

In addition, we give a comment.
\begin{Remark}\label{cha_rem}
From the above arguments and the Rademacher's theorem, we see that the map
\begin{equation*}
\Lambda:(X,Y)\mapsto (t(X,Y), \, x(X,Y))\end{equation*}
 is a.e.~differentiable.  Then we can denote the set $\Omega$ of critical points and the set $ V$ of critical values of $\Lambda$ by
\bel{critp}
\Omega~:=~\Big\{ (X,Y)\,;~~\hbox{either $D\Lambda(X,Y)$ does not exists,
or else $\det\, D\Lambda(X,Y)=0$}\Big\},\eeq
and
$V:=\big\{\Lambda(X,Y)\,;~~(X,Y)\in \Omega\big\}.$
By the area formula \cite{Zi}, the 2-dimensional measure of  $V$ is zero.
We emphasis that the map $\Lambda:\mathbb{R}^2\mapsto\mathbb{R}^2$ is onto but not
one-to-one. However,
for each $(t_0,x_0)\notin V$, there exist a unique point $(X,Y)$ such that $\Lambda(X,Y) = (t_0, x_0)$.

For future use, we record the change of variable formula. For any function  $f\in\L^1(\mathbb{R}^2)$,
the composition
$\tilde f(X,Y) = f(\Lambda(X,Y))$
is well defined at a.e.~point $(X,Y)\in \mathbb{R}^2\setminus\Omega$, thus we obtain
\bel{icv}\int_{\R^2} f(t,x)\, dxdt~=~\int_{\mathbb{R}^2\setminus\Omega} \tilde f(X,Y)
\cdot |\det D\Lambda(X,Y)|\, dXdY,\eeq
here the determinant of the Jacobian matrix $D\Lambda$ is calculated as
\begin{equation*}
\det D\Lambda~=~\det\left(\bega{ccc} t_X & & t_Y\cr
\cr
x_X&& x_Y  \enda\right) ~=~
(\frac{1}{c_2}-\frac{1}{c_1})\alpha
x_X x_Y,
\end{equation*}
with $
x_X=\frac{c_2}{\alpha}t_X$ and $x_Y=\frac{c_1}{\alpha}t_Y\,.$
\end{Remark}

We now introduce more variables. Fixed the initial values $\bar \omega, \bar \upsilon$, and denote $t\mapsto\omega(t,\bar \omega)$ and $t\mapsto \upsilon(t,\bar \upsilon)$ as the unique solutions to \eqref{uweqn}.
We then introduce a new couple of dependent variables $p$ and $q$ as functions of $X$ and $Y$ by
\bel{pdef}
p(X,Y) ~:=~{\partial \over \partial \bar\omega} \omega(\tau,\bar\omega)
\bigg|_{\bar\omega=X,\, \tau = t(X,Y)}\,,\qquad
q(X,Y)~:=~{\partial \over \partial \bar\upsilon} \upsilon(\tau,\bar\upsilon)\bigg|_{\bar\upsilon=Y,
\, \tau = t(X,Y)}\,.\eeq
Moreover, calling that $t\mapsto x^-(t) = x(t,\omega(t))$ and $t\mapsto x^+(t)
=y(t,\upsilon(t))$ are the  unique  backward  and forward characteristics
starting from the points $x(0, \bar\omega)$ and $y(0, \bar\upsilon)$, respectively.
Also, recall the definitions of the maps $\omega\mapsto x(t,\omega)$
and $\upsilon\mapsto y(t,\upsilon)$ in (\ref{xa})--(\ref{yb}), we further introduce a couple of variables describing the feature of characteristics, namely
\bel{nedef}
\sigma(X,Y)~:=~\frac{\partial x}{\partial \omega}(t(X,Y),\omega(t,x(X,Y))),\qquad \eta(X,Y)~:=~\frac{\partial x}{\partial \upsilon}(t(X,Y),\upsilon(t,x(X,Y)))\,.\eeq
In addition, we introduce a new set of variables by setting
\bel{tvt}
\xi(X,Y) ~:=~{{c_2-c_1}\over p(X,Y)} u_{X}(X,Y)\,,\qquad
\zeta(X,Y) ~:=~{{c_2-c_1}\over q(X,Y)} u_{Y}(X,Y)\,.\eeq
Using Rademacher's theorem and Lemma \ref{xy_lem} and Lemmas \ref{nlip_lem}, we see that the above derivatives are a.e.~well defined. Moreover,
$$p(X,Y) ~=~q(X,Y)~=~1\qquad\hbox{if}\qquad t(X,Y)~=~0.$$
Our present purpose is to prove that these variables satisfy the following semi-linear system
with smooth coefficients in $(X,Y)$ coordinates
\bel{sls1}
\begin{cases}
\ds u_{X}=\frac{\xi p}{c_2-c_1},\\
\ds u_{Y}=\frac{\zeta q}{c_2-c_1}, \end{cases}\quad
\begin{cases}\ds x_X=\frac{c_2}{c_2-c_1}\sigma p,\\ \ds x_Y=\frac{c_1}{c_2-c_1}\eta q, \end{cases}\quad
\begin{cases}\ds t_X=\frac{\alpha}{c_2-c_1}\sigma p,\\  \ds t_Y=\frac{\alpha}{c_2-c_1}\eta q,
\end{cases}
\eeq
\bel{sls2}
\begin{cases}
\ds p_Y=\frac{\alpha\partial_x c_1-c_1\partial_x\alpha}{\alpha(c_2-c_1)}pq\eta\sigma+
\frac{2(c_1a_2-c_2a_1)}{c_2(c_2-c_1)}\xi\eta pq
+\frac{2a_1(c_1+c_2)}{c_1(c_2-c_1)}\zeta\sigma pq\\
\ds\qquad-\frac{2c_2a_1}{c_1(c_2-c_1)}\zeta pq-\frac{2c_1a_2}{c_2(c_2-c_1)}\xi pq-\frac{2c_1c_2b}{(c_2-c_1)^2}\xi\zeta pq,\\
\ds q_X=\frac{\alpha\partial_x c_2-c_2\partial_x\alpha}{\alpha(c_2-c_1)}pq\eta\sigma+
\frac{2(c_1a_2-c_2a_1)}{c_1(c_2-c_1)}\sigma\zeta pq
-\frac{2a_2(c_1+c_2)}{c_2(c_2-c_1)}\xi\eta pq\\
\ds\qquad+\frac{2c_2a_1}{c_1(c_2-c_1)}\zeta pq+\frac{2c_1a_2}{c_2(c_2-c_1)}\xi pq+\frac{2c_1c_2b}{(c_2-c_1)^2}\xi\zeta pq,\\
\end{cases}
\eeq
\bel{sls3}
\begin{cases}
\ds \sigma_Y=\frac{\alpha\partial_x c_1-c_1\partial_x\alpha}{\alpha(c_2-c_1)}\eta\sigma(1-\sigma)q
+\frac{2a_1(c_1+c_2)}{c_1(c_2-c_1)}\zeta\sigma(1-\sigma)q\\
\ds\qquad+\frac{2c_1a_2}{c_2(c_2-c_1)}\xi\sigma(1-\eta)q
+\frac{2a_1}{c_2-c_1}\xi\eta(\sigma-1)q+
\frac{2c_1c_2b}{(c_2-c_1)^2}\zeta\xi\sigma q,\\
\ds \eta_X=\frac{\alpha\partial_x c_2-c_2\partial_x\alpha}{\alpha(c_2-c_1)}\eta\sigma(1-\eta)p
+\frac{2a_2(c_1+c_2)}{c_2(c_2-c_1)}\xi\eta(\eta-1)p\\
\ds\qquad+\frac{2c_2a_1}{c_1(c_2-c_1)}\zeta\eta(\sigma-1)p
+\frac{2a_2}{c_2-c_1}\zeta\sigma(1-\eta)p-
\frac{2c_1c_2b}{(c_2-c_1)^2}\zeta\xi\eta p,\\
\end{cases}
\eeq
\bel{sls4}
\begin{cases}
\displaystyle \xi_{Y}=\frac{a_1q}{c_1}(\eta-\sigma\eta)+\frac{a_2q}{c_2}(\sigma-\sigma\eta)
+(a_1-a_2+\frac{2c_2a_1}{c_1})\frac{\xi\zeta q}{c_2-c_1}+\frac{c_2b}{c_2-c_1}\sigma\zeta q\\
\ds\qquad+(d_1+\frac{c_1\partial_x c_2-c_2\partial_xc_1}{c_2-c_1})\frac{\eta\xi q}{c_2-c_1}-\frac{\alpha\partial_x c_1-c_1\partial_x\alpha}{\alpha(c_2-c_1)}\eta\xi\sigma q-\frac{2a_1(c_1+c_2)}{c_1(c_2-c_1)}\xi\sigma\zeta q\\
\ds\qquad-\frac{2(c_1a_2-c_2a_1)}{c_2(c_2-c_1)}\xi^2\eta q+\frac{2c_1a_2}{c_2(c_2-c_1)}\xi^2 q+\frac{2c_1c_2b}{(c_2-c_1)^2}\xi^2\zeta q,\\
\displaystyle \zeta_{X}=\frac{a_1p}{c_1}(\eta-\sigma\eta)+\frac{a_2p}{c_2}(\sigma-\sigma\eta)
+(a_1-a_2-\frac{2c_1a_2}{c_2})\frac{\xi\zeta p}{c_2-c_1}+\frac{c_1b}{c_2-c_1}\xi\eta p\\
\ds\qquad+(d_2+\frac{c_1\partial_x c_2-c_2\partial_x c_1}{c_2-c_1})\frac{\sigma\zeta p}{c_2-c_1}-\frac{\alpha\partial_x c_2-c_2\partial_x\alpha}{\alpha(c_2-c_1)}\eta\zeta\sigma p+\frac{2a_2(c_1+c_2)}{c_2(c_2-c_1)}\xi\eta\zeta p\\
\ds\qquad-\frac{2(c_1a_2-c_2a_1)}{c_1(c_2-c_1)}\zeta^2\sigma p-\frac{2c_2a_1}{c_1(c_2-c_1)}\zeta^2 p-\frac{2c_1c_2b}{(c_2-c_1)^2}\zeta^2\xi p.
\end{cases}
\eeq
Below, we state the main theorem of this section.
\begin{Theorem}\label{equ_thm}
By possibly changing the functions $p,q,\sigma,\eta,\xi,\zeta$
on a set of measure zero in the $X$-$Y$ plane, the following holds.
\begi
\item[(i)] For a.e.~$X_0\in \mathbb{R}$, the functions $t,x,u,p,\sigma,\xi$ are absolutely continuous on every vertical segment of the form $\{ (X_0,Y)\,;~~a<Y<b\}$.   Their partial derivatives w.r.t.~$Y$
satisfy a.e.~the corresponding equations in (\ref{sls1})--(\ref{sls4}).

\item[(ii)] For a.e.~$Y_0\in \mathbb{R}$, the functions $t,x,u,q,\eta,\zeta$ are
absolutely continuous on every horizontal segment of the form
$ \{ (X,Y_0)\,;~~a<X<b\}$.   Their partial derivatives w.r.t.~$X$
satisfy a.e.~the corresponding equations in (\ref{sls1})--(\ref{sls4}).
\endi
\end{Theorem}

In order to achieve this result, we begin with the following technical result, c.f. \cite{BCZ}.
\begin{Lemma}[\cite{BCZ}]\label{tec_lem}
 Let $\Gamma = \, (a,b)\, \times\, (c,d)\,$
be a rectangle  in the $X$-$Y$ plane.

{\rm (i)} Assume that $u\in \L^\infty (\Gamma)$ has a weak partial derivative with respect to $X$. That means $\int_\Gamma (u\varphi_X+f\varphi)\,dX\,dY=0$ for some $f\in L^1(\Gamma)$ and all test functions $\varphi\in C_c^\infty(\Gamma). $ Then, by possibly modifying $u$ on a set of measure zero, the following holds.
For a.e.~$Y_0\in \,(c,d)\,$, the map $X\mapsto u(X, Y_0)$ is absolutely continuous
and
\begin{equation*}
{\partial \over\partial X} u(X,Y_0)~=~f(X, Y_0)\qquad\qquad\hbox{for a.e.~} X\in \,(a,b)\,.
\end{equation*}

{\rm (ii)} Assume that $u\in \L^\infty (\Gamma)$
and $f\in \L^1(\Gamma)$.
Moreover assume that there exists null sets $\N_X\subset\,(a,b)\,$ and
$\N_Y\subset\,(c,d)\,$ such that, for every $\ov X_1,\ov X_2\notin \N_X$ and $\ov Y_1,\ov Y_2\notin \N_Y$ with $\ov X_1<\ov X_2$ and
$\ov Y_1<\ov Y_2$, one has
\begin{equation*}
\int_{\ov Y_1}^{\ov Y_2} \Big[ u(\ov X_2,Y) - u(\ov X_1,Y)\Big]\, dY ~=~\int_{\ov Y_1}^{\ov Y_2}
\int_{\ov X_1}^{\ov X_2} f(X,Y)\, dXdY\,.
\end{equation*}
Then the conclusion of (i) holds.
\end{Lemma}

For future reference, we express the variables $\sigma,\eta,\xi,\zeta$ in terms of $\Tilde{R}$ and $\Tilde{S}$. Indeed, we can prove the following results, using very similar method as in \cite{BCZ}. We omit the proof here for simplicity. For the sake of convenience, we denote a ``good" set in  $(X,Y)$ plane
\begin{equation*}\label{gset}\G ~:=~\mathbb{R}^2\setminus\Omega\,.\end{equation*}
with $\Omega$ defined in \eqref{critp}.
\begin{Lemma}\label{var_lem}

{\rm (i)}  For $(X,Y)\in \G $, it holds that
\begin{equation*}
 \left\{\bega{rl}\displaystyle{p(X,Y)\over x_X(X,Y)} &=\displaystyle\frac{c_2-c_1}{c_2}(1+\Tilde{R}^2),\\
\displaystyle {q(X,Y)\over x_Y(X,Y)}&=\displaystyle\frac{c_2-c_1}{c_1}(1+\Tilde{S}^2),
\enda\right.\quad
\left\{\bega{rl}\sigma(X,Y)&=\displaystyle\frac1{1+\Tilde{R}^2},\\ \eta(X,Y)&=\displaystyle\frac1{1+\Tilde{S}^2}\,,
\enda\right.\quad
\left\{\bega{rl}\xi(X,Y)&=~\displaystyle\frac{R}{1+\Tilde{R}^2},\\ \zeta(X,Y)&=\displaystyle\frac{S}{1+\Tilde{S}^2},
\enda\right.
 \end{equation*}
where the right hand sides are evaluated at the point $(t(X,Y),x(X,Y))$.

{\rm (ii)} For a.e.~$(X,Y)\in\Omega$, one has
$$
\sigma(X,Y)=\eta(X,Y)=0
\quad{\rm and}\quad
\xi(X,Y)=\zeta(X,Y)=0.
$$
\end{Lemma}

With the aforementioned preparation in hand, we now turn to prove Theorem \ref{equ_thm}.

\begin{proof}[\bf Proof of Theorem \ref{equ_thm}.]
In the following, we will concentrate on showing that the variables $t,x,u$, $p,q,$ $\eta,\sigma,\xi,\zeta$ in \eqref{pdef}--\eqref{tvt} indeed satisfy the assumptions of Lemma \ref{tec_lem}.
Toward this goal, consider any rectangle $$\Q:=
[X_1,X_2]\times [Y_1,Y_2],$$
 in the $X$-$Y$ plane, For the sake of clarity, we divide this proof into several
steps.

{\bf (1)-Equations for $u$.} First, recall Lemma \ref{nlip_lem} that the function $u$ is Lipschitz continuous
w.r.t.~ $X,Y$.  Hence, from the definitions (\ref{tvt}), it is easy to see that
$$u_X~=~\frac{1}{c_2-c_1} \, \xi p\,,\qquad u_Y~=~\frac{1}{c_2-c_1} \, \zeta q.$$

{\bf (2)-Equations for $x$ and $t$.}  Using Lemma \ref{nlip_lem} and the definitions (\ref{pdef})--(\ref{nedef}), a direct computation gives rise to
\begin{equation*}
\begin{split}
\frac{\partial}{\partial X}x(X,Y)&=\frac{c_2}{c_2-c_1}\frac{\partial}{\partial\bar\omega}x\Big(t(X,Y),\omega(t(X,Y),\bar\omega)\Big)\bigg|_{\bar\omega=X}\\
&=\frac{c_2}{c_2-c_1}\frac{\partial x}{\partial \omega}\Big(t(X,Y),\omega\big(t(X,Y),x(X,Y)\big)\Big)\frac{\partial\omega}{\partial\bar\omega}(t(X,Y),\bar\omega)\bigg|_{\bar\omega=X}\\
&=\frac{c_2}{c_2-c_1}\sigma p.
\end{split}
\end{equation*}
Furthermore, with the similar argument, we are able to get the equation for $x_Y$. On the other hand, it is clear that
\begin{equation*}
t_X~=~\frac{\alpha}{c_2}x_X~=~\frac{\alpha}{c_2-c_1}\sigma p, \qquad t_Y~=~\frac{\alpha}{c_1}x_Y~=~
\frac{\alpha}{c_2-c_1}\eta q.
\end{equation*}

{\bf (3)-Equations for $p$ and $q$.}
At first, we denote a domain in the $X$-$Y$ plane by
\begin{equation*}
\D:=~\Big\{ (X,Y)\,;\quad X\in  [X_1, \, X_2],\quad  Y\in [Y_1,Y_2],\qquad
 \hbox{det}~D\Lambda(X,Y)\neq0\Big\}.\end{equation*}
 Then the integral equation (\ref{uweqn}) and the change of variable formula (\ref{icv}) implies that
\begin{equation*}
\begin{split}
&\quad\int_{X_1}^{X_2}p(X,Y_2)-p(X,Y_1)\,dX =\int_{X_1}^{X_2} \left[\frac{\partial\omega(\tau,X)}{\partial X}
\bigg|_{\tau=\tau(X,Y_2)}-
\frac{\partial\omega(\tau,X)}{\partial X}\bigg|_{\tau=\tau(X,Y_1)}\right]\,dX\\
& =~\int_{X_1}^{X_2} \left[\frac{\partial}{\partial X}\int_{x(X,Y_1)}^{x(X,Y_2)}\int_{\tau(\Tilde  X,Y_1)}^{\tau(\Tilde X,Y_2)} \Big(\frac{\alpha\partial_x c_1-c_1\partial_x\alpha}{\alpha^2}-\frac{2a_1}{\alpha}(R-S)+G\Big)\,dt\,dx\right]d\Tilde X\\
& =\iint_{\Lambda(\D)}\big(\frac{\alpha\partial_x c_1-c_1\partial_x\alpha}{\alpha^2}-\frac{2a_1}{\alpha}(R-S)+G\big)\,dx\,dt\\
 &=\iint_{\D}
\big(\frac{\alpha\partial_x c_1-c_1\partial_x\alpha}{\alpha^2}-\frac{2a_1}{\alpha}(R-S)+G\big)\cdot |\hbox{det} D\Lambda(X,Y)|\, dX\,dY\\
&=\iint_{\D}
\big(\frac{\alpha\partial_x c_1-c_1\partial_x\alpha}{\alpha^2}-\frac{2a_1}{\alpha}(R-S)+G\big)\frac{\alpha pq }{c_2-c_1}
\frac1{1+\Tilde{R}^2}\frac1{1+\Tilde{S}^2}\,dX\,dY\\
&=\iint_\Q\Big[\frac{\alpha\partial_x c_1-c_1\partial_x\alpha}{\alpha(c_2-c_1)}pq\eta\sigma-\frac{2a_1}{c_2-c_1}(\xi\eta-\zeta\sigma)pq
-\frac{2c_2a_1}{c_1(c_2-c_1)}\zeta pq+\frac{2c_2a_1}{c_1(c_2-c_1)}\sigma\zeta pq\\
&\quad-\frac{2c_1a_2}{c_2(c_2-c_1)}\xi pq+\frac{2c_1a_2}{c_2(c_2-c_1)}\xi \eta pq-\frac{2c_1c_2b}{(c_2-c_1)^2}\xi\zeta pq\Big]\,dX\,dY,
\end{split}
\end{equation*}
with $a_1,a_2,b$ and $G$ defined in \eqref{R-S-eqn} and \eqref{balance}, respectively. Here the last equality follows from Lemma \ref{var_lem}, part (i) for the integral over $\D$ and
part (ii) for the integral over $\Q\setminus \D$.
By using the above equality and Lemma \ref{tec_lem}, we obtain
\bel{pYeq}
\begin{split}
p_Y~=~&\frac{\alpha\partial_x c_1-c_1\partial_x\alpha}{\alpha(c_2-c_1)}pq\eta\sigma+
\frac{2(c_1a_2-c_2a_1)}{c_2(c_2-c_1)}\xi\eta pq
+\frac{2a_1(c_1+c_2)}{c_1(c_2-c_1)}\zeta\sigma pq\\
&-\frac{2c_2a_1}{c_1(c_2-c_1)}\zeta pq-\frac{2c_1a_2}{c_2(c_2-c_1)}\xi pq-\frac{2c_1c_2b}{(c_2-c_1)^2}\xi\zeta pq.
\end{split}\eeq
Applying the same procedure, we can derive the equation for $q_X$.

{\bf (4)-Equations for $\eta$ and $\nu$.}
In view of (\ref{pdef}), (\ref{nedef}) and Remark \ref{cha_rem}, we can arrive at
\begin{equation*}
\begin{split}
&\int_{X_1}^{X_2}\Big[p\sigma(X,Y_2)-p\sigma(X,Y_1)\Big]\,dX
~ =~\ds\int_{X_1}^{X_2} \left[ \frac{\partial x(\tau,X)}{\partial X}
\bigg|_{\tau=t(X,Y_2)}-
\frac{\partial x (\tau,X)}{\partial X}\bigg|_{\tau=t(X,Y_1)}\right]\,dX\\
=&\int_{X_1}^{X_2} \left[\frac{\partial}{\partial X}
\int_{x(X,Y_1)}^{x(X,Y_2)}\int_{t(\Tilde X,Y_1)}^{t(\Tilde X,Y_2)} \Big(\frac{\alpha\partial_x c_1-c_1\partial_x\alpha}{\alpha^2}-\frac{2a_1}{\alpha}(R-S)\Big)\,dt\,dx\right]\,d\Tilde X\\
=&\ds\iint_{\Lambda(\D)}\Big(\frac{\alpha\partial_x c_1-c_1\partial_x\alpha}{\alpha^2}-\frac{2a_1}{\alpha}(R-S)\Big)\,dx\,dt\\
=&\iint_{\D}
\Big(\frac{\alpha\partial_x c_1-c_1\partial_x\alpha}{\alpha^2}-\frac{2a_1}{\alpha}(R-S)\Big)\cdot |\hbox{det} D\Lambda(X,Y)|\, dX\,dY\\
=&\iint_{\D}
\Big(\frac{\alpha\partial_x c_1-c_1\partial_x\alpha}{\alpha^2}-\frac{2a_1}{\alpha}(R-S)\Big)
\frac{\alpha pq}{c_2-c_1}\frac1{1+\Tilde{R}^2}\frac1{1+\Tilde{S}^2} \,dX\,dY\\
=&\iint_{\Q}\big[\frac{\alpha\partial_x c_1-c_1\partial_x\alpha}{\alpha(c_2-c_1)}pq\eta\sigma-\frac{2a_1}{c_2-c_1}(\xi\eta-\zeta\sigma)pq\big]\, dX\,dY.
\end{split}
\end{equation*}
Thus, the
above equality and Lemma \ref{tec_lem} yields
\begin{equation*}
(p\sigma)_Y~=~\frac{\alpha\partial_x c_1-c_1\partial_x\alpha}{\alpha(c_2-c_1)}pq\eta\sigma-\frac{2a_1}{c_2-c_1}(\xi\eta-\zeta\sigma)pq,
\end{equation*}
which together with the equation (\ref{pYeq}) for $p_Y$,
one has the following
equality after a complicated computation
\begin{equation*}
\begin{split}
\sigma_Y=&\frac{\alpha\partial_x c_1-c_1\partial_x\alpha}{\alpha(c_2-c_1)}\eta\sigma(1-\sigma)q
+\frac{2a_1(c_1+c_2)}{c_1(c_2-c_1)}\zeta\sigma(1-\sigma)q\\
&+\frac{2c_1a_2}{c_2(c_2-c_1)}\xi\sigma(1-\eta)q
+\frac{2a_1}{c_2-c_1}\xi\eta(\sigma-1)q+
\frac{2c_1c_2b}{(c_2-c_1)^2}\zeta\xi\sigma q.
\end{split}
\end{equation*}
Similar computation leads to the equation for $\eta_X$

{\bf (5)-Equations for $\xi$ and $\zeta$.} As to the equations for $\xi_Y$ and $\zeta_X$, we turn to establish the distributional derivative $u_{XY}$ firstly.
In light of Lemma \ref{tec_lem}, we indeed want to seek a function $f\in \L^1_{loc}(\R^2)$ such that
\begin{equation*}
[u(X_2, Y_1) - u(X_1, Y_1)]-[u(X_2, Y_2) - u(X_1, Y_2)] ~
=~\int_{X_1}^{X_2}\int_{Y_2}^{Y_1} f(X,Y)\, dX \,dY,\end{equation*}
for any values $X_1<X_2$ and $Y_1> Y_2$.
However, to obtain this equation, one need more subtle estimate on the weak solutions. We will proceed in several steps.
\bigskip
\paragraph{\bf (i).}
In the $t$-$x$ plane, we define the image of points $(X_k,Y_k), k=1,2,3,4$ under the map $\Lambda$ by (c.f. Fig.~ \ref{f:hyp203})
\begin{equation*}\label{Pidef}
\bega{l}
P_1\:= (t_1,x_1)=\Lambda(X_1,Y_1), \qquad P_2\:=(t_2,x_2)=\Lambda(X_2,Y_1),\cr\cr
P_3\:= (t_3,x_3)=\Lambda(X_1,Y_2), \qquad P_4\:=(t_4,x_4)=\Lambda(X_2,Y_2),
\enda\end{equation*}
with
\begin{itemize}
\item Backward characteristic $t\mapsto x_1^-(t)$ passing through $P_1$, $P_3$. (Corresponding to $X=X_1$).
\item Backward characteristic $t\mapsto x_2^-(t)$ passing through $P_2$, $P_4$. (Corresponding to $X=X_2$).
\item Forward characteristic $t\mapsto x_1^+(t)$ passing through $P_1$, $P_2$. (Corresponding to $Y=Y_1$).
\item Forward characteristic $t\mapsto x_2^+(t)$ passing through $P_3$, $P_4$. (Corresponding to $Y=Y_2$).
\end{itemize}
\begin{figure}[htbp]
\centering
\includegraphics[width=0.55\textwidth]{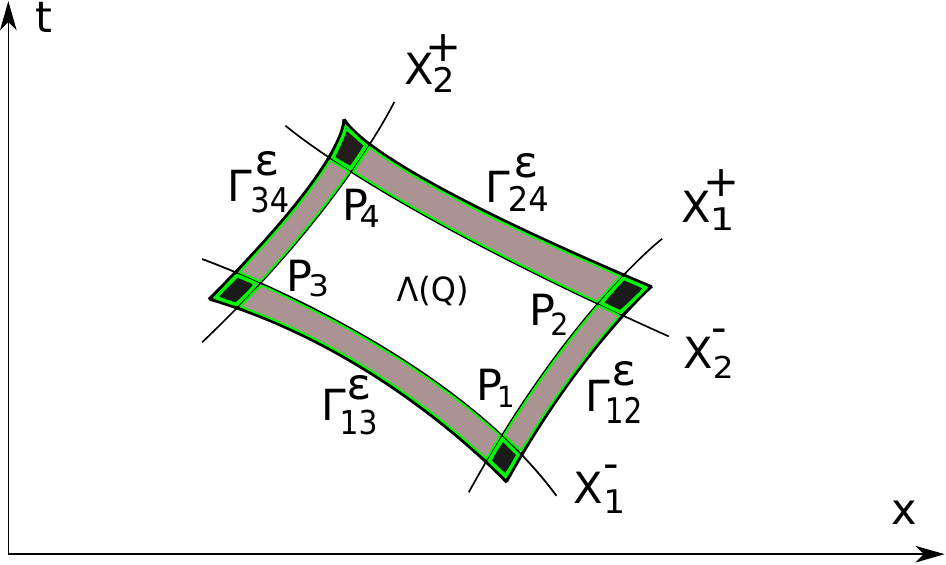}
\includegraphics[width=0.4\textwidth]{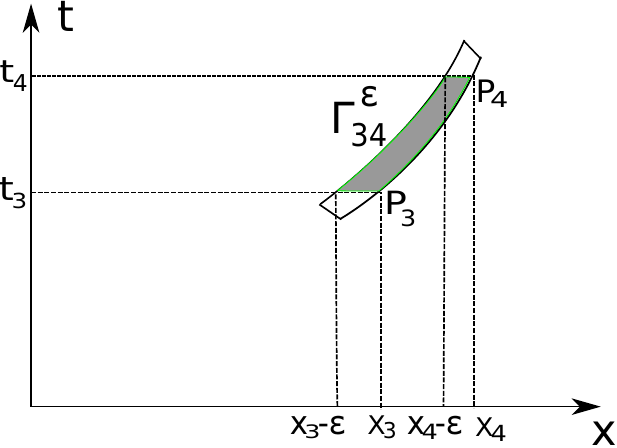}
\caption{\small   Left: The support of the test function $\phi^\epsilon$ in (\ref{pen11}). Right: An enlarged picture of $\Gamma_{34}^\ve$, which is used in \eqref{lime2} where we only do the calculation in the shaded region, becuase the unshaded region can be omitted as $\ve\rightarrow 0$.}
\label{f:hyp203}
\end{figure}

We investigate a family of test functions $\phi^\epsilon$ approaching the
characteristic function of the set $\Lambda(\Q)$, where $\Q\:= [X_1,X_2]\times [Y_2,Y_1]$.
More precisely, set
 \bel{pen11}\phi^{\epsilon}(s,y)~\:=~\min\{ \varrho^{\epsilon}(s,y),
\,\varsigma^\epsilon(s,y)\},\eeq
where
\bel{testf1}\varrho^{\epsilon}(s,y)~\:=~\left\{\bega{cl} 0 \qquad &\hbox{if}\quad y~
\leq ~ x_1^-(s)-\epsilon,\cr
1+\epsilon^{-1}(y-x_1^-(s)) \qquad &\hbox{if}\quad x_1^-(s)-\epsilon~\leq~ y\leq
x_1^-(s),\cr
1\qquad &\hbox{if}\quad x_1^-(s)~ \leq ~y~\leq~ x_2^-(s),\cr
1-\epsilon^{-1}(y-x_2^-(s))\qquad &\hbox{if}\quad  x_2^-(s)~\leq ~
y~\leq~ x_2^-(s)+\epsilon,\cr
0\qquad &\hbox{if}\quad y~\geq ~x_2^-(s)+\epsilon\,,\enda\right.\eeq
and
\bel{testf2}\varsigma^\epsilon(s,y)~\:=~\left\{\bega{cl} 0 \qquad &\hbox{if}\quad y
~\leq~  x_2^+(s)-\epsilon,\cr
1-\epsilon^{-1}(y-x_2^+(s)) \qquad &\hbox{if}\quad x_2^+(s)-\epsilon~\leq~ y~\leq~
x_2^+(s),\cr
1\qquad &\hbox{if}\quad x_2^+(s)~ \leq~ y~\leq~ x_1^+(s),\cr
1+\epsilon^{-1}(y-x_1^+(s))\qquad &\hbox{if}\quad  x_1^+(s)~\leq~ y~\leq ~
x_1^+(s)+\epsilon,
\cr
0\qquad &\hbox{if}\quad y~\geq~ x_1^+(s)+\epsilon\,.\enda\right.\eeq
In view of $(\ref{R-S-eqn})_1$, it is easy to see that
\begin{equation*}
\begin{split}
\iint R\Big[(\alpha\varphi)_t +(c_1\varphi)_x\Big] \,dx\,dt~=&~
-\iint \big[a_1 R^2-(a_1+a_2)RS+a_2S^2+c_2bS-d_1R\big]\varphi\, dx\,dt\\
\end{split}\end{equation*}
for every test function $\varphi\in \C^1_c(\R^2)$.  By the conditions \eqref{con},
we are able to choose a sequence of test functions $\vp_n$ such that,
$$\vp_n~\to~\frac{\phi^\epsilon}{\alpha(c_2-c_1)}\quad \hbox{in}~~H^1(\mathbb{R}^2),\quad {\rm as }\quad n\to\infty,$$
with $\phi^\epsilon$ defined in \eqref{pen11}. Taking the limit, we have
 \begin{equation*}
\begin{split}
&\iint R\big[ \big( {\phi^\epsilon\over c_2-c_1}\big)_t+ (\frac{c_1\phi^\epsilon}{\alpha(c_2-c_1)})_x\big]\,dx\,dt\\
=&
-\iint [a_1 R^2-(a_1+a_2)RS+a_2S^2+c_2bS-d_1R]\frac{\phi^\epsilon}{\alpha(c_2-c_1)}\, dx\,dt.
\end{split}\end{equation*}
By a elaborate calculation, it holds that
 \begin{equation}\label{R+S}
\begin{split}
&\iint {R\over \alpha(c_2-c_1)}( \alpha\phi^\epsilon_t+ c_1\phi^\epsilon_x)\,dx\,dt\\
=&
\iint \big[a_1 R^2+(a_1-a_2+\frac{\partial_u\alpha}{\alpha})RS-a_2S^2-c_2bS-(d_1-\partial_x c_1)R\big]\frac{\phi^\epsilon}{\alpha(c_2-c_1)} \, dx\,dt.\\
\end{split}\end{equation}

\paragraph{\bf (ii).}
By the definition of $\phi^\epsilon$ in \eqref{pen11}--\eqref{testf2}, we see that $\alpha\phi^\epsilon_t+ c_1\phi^\epsilon_x$ is supported on a small neighborhood of the boundary of $\Lambda(\mathcal{Q})$. Thus, in what follows, we will focus on the four boundary strips $\Gamma^\epsilon_{12}$, $\Gamma^\epsilon_{13}$, $\Gamma^\epsilon_{24}$ and $\Gamma^\epsilon_{34}$ of the support of $\phi^\epsilon$ in Fig.~\ref{f:hyp203}.
For example, $\Gamma^\epsilon_{34}$ is the strip enclosed by $x_2^+(t)-\epsilon$, $x_2^+(t)$, $x_1^-(t)-\ve$
and $x_2^-(t) +\epsilon$.
These sets  overlap near the points $P_i= (t_i, x_i)$, $i=1,2,3,4$.
Moreover, each of these intersections is contained in a ball of radius $\O(\ve)$.
For example, $\Gamma_{12}^\epsilon\cap \Gamma_{13}^\epsilon~\subset~B(P_1, K\epsilon),
$ for some constant $K$ and all $\epsilon>0$. We begin with some estimates on theses intersections, which basically says that these regions can be omitted.
Indeed, one has
\begin{equation*}
\begin{split}
&\quad\lim\limits_{\epsilon\to 0}\left|\iint_{\Gamma_{12}^\epsilon\cap\Gamma_{13}^\epsilon}\frac {R}{\alpha(c_2-c_1)}(\alpha\phi_t^\epsilon+c_1\phi^\epsilon_x)\,dx\,dt\right|~
 \leq ~\ds\lim\limits_{\epsilon\to 0}\mathcal{O}(1)\cdot\frac{1}{ \epsilon} \iint_{B(P_1, K\epsilon)}|R|\, dx\,dt
\\
& \leq ~\mathcal{O}(1)\cdot\lim\limits_{\epsilon\to 0}\frac{1}{ \epsilon}
\int_{t_1-K\epsilon}^{t_1+K\epsilon}
\Big(\int_{x_1-K\epsilon}^{x_1+K\epsilon}  R^2(t,x) dx\Big)^{1/2}
(2K\epsilon)^{1/2} \, dt\\
&\leq ~ \mathcal{O}(1)\cdot\lim\limits_{\epsilon\to 0} \frac{1}{ \epsilon}
E_0^{1/2}(2K\epsilon)^{3/2}=0.
\end{split}
\end{equation*}
Applying the same procedure for the other three intersections, we can get
\bel{thus}
\begin{split}
&\lim\limits_{\epsilon\to 0}\iint {R\over \alpha(c_2-c_1)}({\alpha\phi^\epsilon_t+c_1\phi^\epsilon_x})\,
dx\,dt\\
&=~\lim\limits_{\epsilon\to 0}\iint_{\Gamma^\epsilon_{12}\cup\Gamma^\epsilon_{13}\cup\Gamma^\epsilon_{24}\cup\Gamma^\epsilon_{34}}{R\over \alpha(c_2-c_1)}({\alpha\phi^\epsilon_t+c_1\phi^\epsilon_x})\,dx\,dt\,\\
&=~\lim\limits_{\epsilon\to 0}\left(\iint_{\Gamma^\epsilon_{12}}+\iint_{\Gamma^\epsilon_{13}}+\iint_{
\Gamma^\epsilon_{24}}+\iint_{\Gamma^\epsilon_{34}}\right)
{R\over \alpha(c_2-c_1)}({\alpha\phi^\epsilon_t+c_1\phi^\epsilon_x})\,dx\,dt\,.
\end{split}\eeq

\paragraph{\bf(iii).}
Now, it remains to bound the integral over the four boundary strips. As for  the integral over $\Gamma_{13}^\epsilon$, the Cauchy's inequality and  the definition of $\phi^\epsilon$ in \eqref{pen11}--(\ref{testf2}) implies
\bel{13Ge}
\begin{split}
&~\quad\lim\limits_{\epsilon\to 0}\iint_{\Gamma^\epsilon_{13}}
{R\over \alpha(c_2-c_1)}({\alpha\phi^\epsilon_t+c_1\phi^\epsilon_x})\,dx\,dt\,\\
&=~\lim_{\epsilon\to 0}\iint_{\Gamma_{13}^\epsilon}\frac{R(t,x)}{\alpha(c_2-c_1)}
\cdot \frac{-c_1\big(x_1^-(t),u(t,x_1^-(t))\big)+c_1\big(x,u(t,x)\big)}{\epsilon}\, dx\,dt\\
&\leq~ \O(1)\cdot\lim\limits_{\epsilon\to0}\iint_{\Gamma_{13}^\epsilon}  \frac{|-c_1\big(x_1^-(t),u(t,x_1^-(t))\big)+c_1\big(x,u(t,x)\big)|}{\epsilon} |R(t,x)|dx\,dt\\
&\leq~\O(1)\cdot \lim\limits_{\epsilon\to0}\frac{1}{\epsilon}\int_{0}^{T} \int_{0\leq x^-_1(t)-x\leq \ve} \big[|x^-_1(t)-x|+|x^-_1(t)-x|^{1/2}\big]\, |R(t,x)|\, dx\,dt \\
&\leq~ \O(1)\cdot\lim\limits_{\epsilon\to0}\frac{1}{\epsilon}\int_{0}^{T}  \left(\int_{0\leq x^-_1(t)-x\leq\ve}|R(t,x)|^2\, dx
\right)^{1/2}(\epsilon^\frac{3}{2}+\epsilon)\,dt=0\,.
\end{split}
\eeq
Repeating this argument for the integral over $\Gamma_{24}^\epsilon$, we have
\bel{iGe2}
\lim_{\epsilon\to 0}~\iint_{\Gamma_{24}^\epsilon}
{R\over \alpha(c_2-c_1)}({\alpha\phi^\epsilon_t+c_1\phi^\epsilon_x})\,dx\,dt ~=~0\,.\eeq
Next, we turn to the integral over $\Gamma_{12}^\epsilon$, it holds that
\begin{equation}
\begin{split}
&~\quad\lim_{\epsilon\to 0}~\iint_{\Gamma_{12}^\epsilon}
{R\over \alpha(c_2-c_1)}({\alpha\phi^\epsilon_t+c_1\phi^\epsilon_x})\,dx\,dt\\
&=~\lim_{\epsilon\to 0}\iint_{\Gamma_{12}^\epsilon}\frac{R(t,x)}{\alpha(c_2-c_1)}
\cdot \frac{-c_2\big(x_1^+(t),u(t,x_1^+(t))\big)+c_1\big(x,u(t,x)\big)}{\epsilon}\, dx\,dt\\
& =~\ds \lim\limits_{\epsilon\to0}\iint_{\Gamma_{12}^\epsilon}  \frac{R(t,x)}{\alpha(c_2-c_1)}
\frac{-c_2\big(x_1^+(t),u(t,x_1^+(t))\big)+c_2\big(x,u(t,x)\big)}{\epsilon} dx\,dt\\
&\quad\quad-\lim_{\epsilon\to 0}\iint_{\Gamma_{12}^\epsilon}
\frac{R(t,x)}{\alpha(c_2-c_1)}  \,\frac{c_2\big(x,u(t,x)\big)-c_1\big(x,u(t,x)\big)}{\epsilon}\, dx\,dt\\
& =~\ds -\lim_{\epsilon\to0}\frac1\epsilon\iint_{\Gamma_{12}^\epsilon}\frac{R(t,x)}{\alpha\big(x,u(t,x)\big)}\,dx\,dt\,,
\end{split}
\end{equation}
the last equality follows from the same argument as in \eqref{13Ge}. Similarly, for the integral over $\Gamma_{34}^\epsilon$, we can derive
\bel{iGe}
\lim_{\epsilon\to 0}~\iint_{\Gamma_{34}^\epsilon}
{R\over\alpha(c_2-c_1)}({\alpha\phi^\epsilon_t+c_1\phi^\epsilon_x})\,dx\,dt ~=~\ds \lim_{\epsilon\to0}\frac1\epsilon\iint_{\Gamma_{34}^\epsilon}{R(t,x)\over \alpha\big(x,u(t,x)\big)}\,dx\,dt\,.\eeq
Moreover, we can get from the fact $\bigl\{a_1 R^2+(a_1-a_2+\frac{\partial_u\alpha}{\alpha})RS-a_2S^2-c_2bS-(d_1-\partial_x c_1)R \bigr\}\in \L^1_{loc}(\mathbb{R}^2)$ that
\begin{equation}\label{lsource}
\begin{split}
&\quad\lim_{\epsilon\to 0}\iint\big[a_1 R^2+(a_1-a_2+\frac{\partial_u\alpha}{\alpha})RS-a_2S^2-c_2bS-(d_1-\partial_x c_1)R\big]\frac{\phi^\epsilon}{\alpha(c_2-c_1)} \, dx\,dt\\
&~=~\iint_{\Lambda(\Q)}\frac{1}{\alpha(c_2-c_1)}\big[a_1 R^2+(a_1-a_2+\frac{\partial_u\alpha}{\alpha})RS-a_2S^2-c_2bS-(d_1-\partial_x c_1)R\big] \, dx\,dt\,.
\end{split}\end{equation}
Plugging the estimates \eqref{thus}--(\ref{lsource}) into (\ref{R+S}), we then obtain
\begin{equation*}\label{lime2}
\begin{split}
&\qquad\lim_{\epsilon\to0}\frac1\epsilon\iint_{\Gamma_{34}^\epsilon}{R\over\alpha}\,dx\,dt-
\lim_{\epsilon\to0}\frac1\epsilon\iint_{\Gamma_{12}^\epsilon}{R\over\alpha}\,dx\,dt\\
&~=~\iint_{\Lambda(\Q)}\frac{1}{\alpha(c_2-c_1)}\big[a_1 R^2+(a_1-a_2+\frac{\partial_u\alpha}{\alpha})RS-a_2S^2-c_2bS-(d_1-\partial_x c_1)R\big]\,dx\,dt\,.
\end{split}
\end{equation*}

\paragraph{\bf(iv).}
On the other hand, observe that $u\in H^1_{loc}$, we further have that  (Fig.~\ref{f:hyp203}, right)
\begin{equation*}
\begin{split}
u(P_2)-u(P_1)&=~\lim_{\ve\to 0} {1\over\epsilon}
 \left( \int_{x_2-\epsilon}^{x_2}
 u(t_2, y)\, dy -
 \int_{x_1-\epsilon}^{x_1} u(t_1, y)\, dy\right)\\
 &=~\lim_{\epsilon\to 0}
{1\over\epsilon}\dint_{\Gamma_{12}^\epsilon} \Big[u_{t} + \frac{c_2}{\alpha}\big(x_1^+(t),u(t, x_1^+(t))\big)u_{x}
\Big]\, dx\,dt\\
&=~
\lim_{\epsilon\to 0}{1\over\epsilon}
\dint_{\Gamma_{12}^\epsilon} \Big[{\alpha u_{t} + c_2(x,u(t, x))u_{x}\over \alpha}\Big]\, dx\,dt\\
&=~
\lim_{\epsilon\to 0}{1\over\epsilon}
\dint_{\Gamma_{12}^\epsilon} {R\over \alpha}\, dx\,dt\,.
\end{split}
\end{equation*}
Using a very similar argument for $u(P_4)-u(P_3)$, we thus conclude
\bel{es5}
\begin{split}
&\quad[u(P_4)-u(P_3)]-[u(P_2)-u(P_1)]\\
&=~\iint_{\Lambda(\Q)}\frac{1}{\alpha(c_2-c_1)}\big[a_1 R^2+(a_1-a_2+\frac{\partial_u\alpha}{\alpha})RS-a_2S^2-c_2bS-(d_1-\partial_x c_1)R\big]\, dx\,dt\,.
\end{split}
\eeq
Here, in view of Remark \ref{cha_rem}, we can write the right hand side of (\ref{es5}) as an integral w.r.t.~the variables $X,Y$, that is
\begin{equation*}
\begin{split}
&\qquad [u(X_2, Y_2) - u(X_1, Y_2)]-[u(X_2, Y_1) -u(X_1, Y_1)] \\
& =~ \iint_{\Q\cap \G }\frac{1}{(c_2-c_1)^2}\big[a_1 R^2+(a_1-a_2+\frac{\partial_u\alpha}{\alpha})RS-a_2S^2-c_2bS\big]\cdot \frac{p}{1+\Tilde{R}^2}\frac{q}{1+\Tilde{S}^2}\,dXdY\\
&\quad -\iint_{\Q\cap \G }\frac{(d_1-\partial_x c_1)R}{(c_2-c_1)^2}\cdot \frac{p}{1+\Tilde{R}^2}\frac{q}{1+\Tilde{S}^2}\,dXdY\,.\end{split}\end{equation*}
Hence, by using Lemma  \ref{tec_lem}, we konw that the weak derivative $u_{XY}$
exists and if $\det D\Lambda(X,Y)\neq0$, then
\begin{equation*}
\begin{split}
u_{XY}(X,Y)~=~&-\frac{1}{(c_2-c_1)^2}\big[a_1 R^2+(a_1-a_2+\frac{\partial_u\alpha}{\alpha})RS-a_2S^2-c_2bS\big]\cdot \frac{p}{1+\Tilde{R}^2}\frac{q}{1+\Tilde{S}^2}\\
&+\frac{(d_1-\partial_x c_1)R}{(c_2-c_1)^2}\cdot \frac{p}{1+\Tilde{R}^2}\frac{q}{1+\Tilde{S}^2},
\end{split}\end{equation*}
if $\det D\Lambda(X,Y) =0$, then
$$u_{XY}(X,Y)~=~0.$$
Therefore, a
direct calculation implies that
\bel{xpY}\begin{split}
u_{XY}~=~&\big(\frac{a_1}{c_1}(\eta-\sigma\eta
)+\frac{a_2}{c_2}(\sigma-\sigma\eta)\big)\frac{pq}{c_2-c_1}
-(a_1-a_2+\frac{\partial_u\alpha}{\alpha})\frac{\xi\zeta pq}{(c_2-c_1)^2}\\
&+\big(c_2b\sigma\zeta+(d_1-\partial_x c_1)\eta\xi\big)\frac{pq}{(c_2-c_1)^2}\,.
\end{split}\eeq

\paragraph{\bf(v).}
With the above preparation in hand, we are ready to derive the equation for $\xi_{Y}$. By Lemma \ref{tec_lem}, we see that, for a.e.~$X$, the map
$Y\mapsto  u_{X}(X,Y)~=~\frac{\xi p}{c_2-c_1}(X,Y)$ is absolutely continuous and its derivative is
given by (\ref{xpY}).
In view of the equations for $p_Y, u_Y$ and $x_Y$, and the fact that $p$ remains uniformly positive on bounded sets,
one has the following
equality after a complicated computation
\begin{equation*}
\begin{split}
\xi_{Y}~=~&\frac{a_1q}{c_1}(\eta-\sigma\eta)+\frac{a_2q}{c_2}(\sigma-\sigma\eta)
+(a_1-a_2+\frac{2c_2a_1}{c_1})\frac{\xi\zeta q}{c_2-c_1}+\frac{c_2b}{c_2-c_1}\sigma\zeta q\\
&+(d_1+\frac{c_1\partial_x c_2-c_2\partial_xc_1}{c_2-c_1})\frac{\eta\xi q}{c_2-c_1}-\frac{\alpha\partial_x c_1-c_1\partial_x\alpha}{\alpha(c_2-c_1)}\eta\xi\sigma q-\frac{2a_1(c_1+c_2)}{c_1(c_2-c_1)}\xi\sigma\zeta q\\
&-\frac{2(c_1a_2-c_2a_1)}{c_2(c_2-c_1)}\xi^2\eta q+\frac{2c_1a_2}{c_2(c_2-c_1)}\xi^2 q+\frac{2c_1c_2b}{(c_2-c_1)^2}\xi^2\zeta q.
\end{split}\end{equation*}
The equation for $
\zeta_{X}$ can be established in a similar argument. This completes the proof of Theorem \ref{equ_thm}. \end{proof}

\section{Proof of Theorem \ref{thm_un}}

Now we are ready to prove Theorem \ref{thm_un} on the uniqueness of conservative
solutions to the system (\ref{vwl})--\eqref{ID}.
Let initial data $u_{0}\in H^1(\mathbb{R})$, $u_{1}\in\L^2(\mathbb{R}) $ be given.
These data uniquely determine a curve $\gamma$
in the $X$-$Y$ plane, parameterized by
$$X(x):= ~x+ \int_{-\infty}^x \Tilde{R}^2(0,y)\, dy\,,\qquad\qquad
Y(x):= ~x+ \int_{-\infty}^x \Tilde{S}^2(0,y)\, dy.$$
At the point $(X(x), Y(x))\in\gamma$, we have
$$\left\{\bega{rl}  t&=0  ,\cr x&=~x\,,\enda\right.\qquad\qquad
\left\{\bega{rl}  u&=~u_{0}(x),\cr p&=~q~=~1\,,\enda\right.
$$
$$\left\{\bega{rl}
\sigma &=~{1\over 1+\Tilde{R}^2(0,x)}\,,\cr   \eta&=~{1\over 1+\Tilde{S}^2(0,x)}\,,
\enda\right.
\qquad\qquad \left\{\bega{rl}\xi&=~{R(0,x)\over 1+\Tilde{R}^2(0,x)}\,,\cr
\zeta&=~
{S(0,x)\over 1+\Tilde{S}^2(0,x)}\,,\enda\right.$$
with
$$R(0,x)=\alpha\big(x,u_0(x)\big) u_1(x) + c_2\big(x,u_{0}(x)\big) u_{0,x},\quad
S(0,x)=\alpha\big(x,u_0(x)\big) u_1(x) + c_1\big(x,u_{0}(x)\big) {u}_{0,x},$$
and
$$\Tilde{R}^2(0,x)=\frac{-c_1}{c_2-c_1}\big(x,u_0(x)\big)R^2(0,x),\quad \Tilde{S}^2(0,x)=\frac{c_2}{c_2-c_1}\big(x,u_0(x)\big) S^2(0,x).$$
By using an argument analog to the one in \cite{H} for a semi-linear system,
we can obtain a unique solution $(t,x,u$, $p,q,$ $\sigma,\eta,\xi,\zeta)$ of the semi-linear system (\ref{sls1})--(\ref{sls4}) with the above boundary data along $\gamma$ in the $X$-$Y$ plane. Moreover, the functions
$(X,Y)~\mapsto~(x,t,u)(X,Y)$
are  uniquely determined, up to a set of zero measure in the $X$-$Y$ plane.
Since the map  $(t,x)\mapsto u(t,x)$ is continuous,
we thus conclude that $u(t,x)$ is  uniquely determined, pointwise in the $x$-$t$ plane.
This completes the proof of  Theorem \ref{thm_un}.
\endproof

%

\section*{Acknowledgements}
The first author is supported by the National Natural Science Foundation of China (No. 11801295) and the Shandong Provincial Natural Science Foundation, China (No. ZR2018BA008).  The second author is partially
supported by National Science Foundation with grants DMS-1715012 and DMS-2008504. The third author is supported by the National Natural Science Foundation of China (No. 11971199) and Guandong Provincial Natural Science Foundation of China (No. 2020B1515310012)

%


\begin{thebibliography}{99}
 \bibitem{AH2007} G. Ali and J. K. Hunter, Diffractive nonlinear geometrical optics for variational wave equations and the Einstein equations, {\it Comm. Pure Appl. Math.}, {\bf 60} (2007), 1522--1557.

\bibitem{AH} G. Ali and J. Hunter, Orientation waves in a director field
with rotational inertia, {\it Kinet. Relat. Models}, {\bf 2} (2009), 1--37.



\bibitem{B1}
A.~Bressan, Uniqueness of conservative solutions for nonlinear wave equations via characteristics, {\it Bull. Braz. Math. Soc.} {\bf 47}(1) (2016), 157--169.


\bibitem{BC}
A. Bressan and G. Chen,  Generic regularity of conservative solutions to a nonlinear wave equation, {\it  Ann. Inst. H.~Poincar\'e Anal. Non Lin\'eaire} {\bf 34}(2) (2017), 335--354.

\bibitem{BC2015}
A. Bressan and G. Chen, Lipschitz metrics for a class of nonlinear wave equations,  {\it Arch. Ration. Mech. Anal.} {\bf 226}(3) (2017), 1303--1343.


\bibitem{BCZ2} A. Bressan, G. Chen and Q. Zhang, Uniqueness of conservative solutions to the Camassa-Holm equation via characteristics, {\it Discr. Cont. Dynam. Syst.} {\bf 35} (2015), 25--42.

\bibitem{BCZ} A. Bressan, G. Chen and Q. Zhang,
Unique conservative solutions to a variational wave equation, {\it Arch. Ration. Mech. Anal.} {\bf 217} (3) (2015), 1069--1101.

 \bibitem{BZ}
A.~Bressan and Y.~Zheng,
Conservative solutions to a nonlinear variational wave equation,
{\it Comm. Math. Phys.} {\bf 266} (2006), 471--497.

\bibitem{BH}A.~Bressan and T.~Huang,
Representation of dissipative solutions to a nonlinear variational wave equation,
{\it Comm. Math. Sci.} {\bf 14} (2016), 31--53.

\bibitem{BHY}
A.~Bressan, T.~Huang  and F.~Yu,
Structurally stable singularities for a nonlinear wave equation, {\it Bull. Inst. Math.
Acad. Sinica.} {\bf 10}(4) (2015), 449--478.


\bibitem{CCD} H. Cai, G. Chen and Y. D, Uniqueness and regularity of conservative solution to a wave system modeling nematic liquid crystal, {\it J. Math. Pures Appl.} {\bf 117} (2018), 185--220.

\bibitem{CCS} H. Cai, G, Chen and Y. Shen,  A Finsler type Lipschitz optimal transport metric for a quasilinear wave equation, submitted, available at arXiv:2007.15201.

 \bibitem{CZ12} G. Chen and Y. Zheng,
Singularity and existence for a wave system of nematic liquid crystals, {\it J. Math. Anal. Appl.} {\bf 398} (2013), 170--188.

\bibitem{CZZ12} G. Chen, P. Zhang and Y. Zheng,
 Conservation solutions to a system of variational wave equations of nematic liquid crystals,
{\it Commun. Pure Appl. Anal. } {\bf 12}(3) (2013), 1445--1468.

\bibitem{GHZ} R.T. Glassey, J.K. Hunter and Y. Zheng, Singularities in a nonlinear variational wave equation, {\it J. Differential. Equations}
{\bf 129}, (1996),  49--78.

\bibitem{GHZ2} R.T. Glassey,  J.K. Hunter and  Y. Zheng, Singularities and oscillations in a nonlinear variational wave equation. Singularities and Oscillations, edited by J. Rauch, M.E. Taylor, (eds.) IMA, Vol. 91, Springer, 1997.



\bibitem{HR} H.~Holden and X.~Raynaud, 
Global semigroup of conservative solutions of the nonlinear variational wave equation.
{\em Arch. Ration. Mech. Anal.} {\bf 201} (2011),  871-964.

\bibitem{H} Y.B. Hu, Conservative solutions to a one--dimensional nonlinear variational wave equation, {\it J. Differential Equations} {\bf 259} (2015), 172--200.




\bibitem{RS} I. Rodnianski and J. Sterbenz, On the formation of singularities in the critical $O(3)\sigma$-model, {\it Ann. of Math.} {\bf 172} (2010), 187--242.

\bibitem{S} R.A. Saxton, Dynamic instability of the liquid crystal director, In: Contemporary Mathematics, Vol. 100: Current Progress in Hyperbolic Systems, ed. W.B. Lindquist, Providence RI: AMS, 1989, 325-¨C330.

    \bibitem{ZZ03} P. Zhang and Y. Zheng,
Weak solutions to a nonlinear variational wave equation,
{\em Arch. Ration. Mech. Anal.} {\bf 166} (2003), 303--319.

    \bibitem{ZZ10} P. Zhang and Y. Zheng, Conservative solutions to
a  system of variational wave equations of nematic liquid crystals,
{\em Arch. Ration. Mech. Anal.} {\bf 195} (2010), 701--727.

\bibitem{ZZ11} P. Zhang and Y. Zheng,
Energy conservative solutions to a one-dimensional full variational wave system,
{\em Comm. Pure Appl. Math.} {\bf 55} (2012), 582--632.

\bibitem{Zi} W.P. Ziemer, Weakly Differentiable Functions. Sobolev Spaces and Functions of Bounded Variation, Grad. Texts Math., vol.120, Springer-Verlag, New York, 1989.

\end{thebibliography}
\end{document}